\documentclass[a4paper]{article}
\usepackage[utf8]{inputenc}
\usepackage[T1]{fontenc}
\usepackage{amsmath,amssymb,amsthm,mathrsfs}
\usepackage{cite}
\usepackage{mathtools}
\usepackage{enumitem}

\setlist[enumerate,1]{label=\arabic*), ref=\arabic*}

\usepackage[pdftex,bookmarksopen=true,bookmarks=true,unicode,setpagesize]{hyperref}
\hypersetup{colorlinks=true,linkcolor=black,citecolor=black,urlcolor=blue}

\let\equation=\gather
\let\endequation=\endgather

\allowdisplaybreaks[3]
\numberwithin{equation}{section}

\makeatletter
\renewcommand*{\@fnsymbol}[1]{\ensuremath{\ifcase#1\or 1\or 2\else\@ctrerr\fi}}
\makeatother

\newtheorem{theorem}{Theorem}
\newtheorem{corollary}{Corollary}[section]
\newtheorem{lemma}[corollary]{Lemma}
\newtheorem{proposition}[corollary]{Proposition}
\theoremstyle{definition}
\newtheorem{definition}[corollary]{Definition}
\newtheorem{example}[corollary]{Example}
\theoremstyle{remark}
\newtheorem{remark}[corollary]{Remark}

\DeclareMathOperator*{\esssup}{ess\,sup}
\DeclareMathOperator*{\essinf}{ess\,inf}

\newcommand{\R}{\mathbb{R}}
\newcommand{\X}{{\mathbb{R}}}
\newcommand{\N}{\mathbb{N}}

\newcommand{\eps}{\varepsilon}
\newcommand{\la}{\lambda}
\newcommand{\La}{\Lambda}
\newcommand{\ga}{\gamma}
\newcommand{\ka}{\varkappa}

\newcommand{\M}{{\mathcal{M}}}
\renewcommand{\L}{{\mathcal{L}}}
\renewcommand\S{\mathcal{S}}
\renewcommand{\P}{{\mathcal{P}}}

\newcommand\MS{\mathcal{MS}}
\newcommand{\PL}{{\mathcal{PL}}}
\newcommand\F[1][]{\mathcal{F}_{#1}}

\newcommand{\ma}{\mathfrak{m}}

\newcommand{\lt}{l_\theta}
\newcommand{\1}{1\!\!1}
\newcommand{\locun}{\xRightarrow{\,\mathrm{loc}\ }}
\newcommand{\diff}[2]{\frac{\partial #1}{\partial #2}}

\newcounter{assum}
\newenvironment{assum}{\refstepcounter{assum}\let\myabovedisplayskip=\abovedisplayskip\abovedisplayskip=8pt plus 2pt minus 2pt\let\mybelowdisplayskip=\belowdisplayskip\belowdisplayskip=8pt plus 2pt minus 2pt\equation\tag{\ensuremath{\mathrm{A}\theassum}}}{\endequation\let\abovedisplayskip=\myabovedisplayskip\let\belowdisplayskip=\mybelowdisplayskip}

\title{Accelerated nonlocal nonsymmetric dispersion for monostable equations on the real line}

\author{Dmitri Finkelshtein\thanks{Department of Mathematics,
Swansea University, Singleton Park, Swansea SA2 8PP, U.K. ({\tt d.l.finkelshtein@swansea.ac.uk}).} \and Pasha Tkachov\thanks{Fakult\"{a}t
f\"{u}r Mathematik, Universit\"{a}t Bielefeld, Postfach 110 131, 33501 Bielefeld,
Germany ({\tt ptkachov@math.uni-bielefeld.de}).}}

\begin{document}

\maketitle 

\begin{abstract}
We consider the accelerated propagation of solutions to equations with a nonlocal linear dispersion on the real line and monostable nonlinearities (both local or nonlocal, however, not degenerated at $0$), in the case when either of the dispersion kernel or the initial condition has regularly heavy tails at both $\pm\infty$, perhaps different. We show that, in such case, the propagation to the right direction is fully determined by the right tails of either the kernel or the initial condition. We describe both cases of integrable and monotone initial conditions which may give different orders of the acceleration. Our approach is based, in particular, on the extension of the theory of sub-exponential distributions, which we introduced early in \cite{FT2017b}. 

\textbf{Keywords:} nonlocal diffusion; reaction-diffusion equation; front propagation; acceleration; monostable equation; nonlocal nonlinearity; long-time behavior; integral equation

\textbf{2010 Mathematics Subject Classification:} 35B40, 35K57, 47G20, 45G10
\end{abstract}

\section{Introduction}
We will study non-negative solutions $u:\X\times\R_+\to\R_+:=[0,\infty)$ to the equation
\begin{equation}\label{eq:basicequation}
\begin{cases}
\dfrac{\partial}{\partial t} u(x,t) = \ka \displaystyle \int_\X a(x-y)u(y,t)\,dx-m u(x,t)-u(x,t) (Gu)(x,t), \\[3mm]
u(x,0)=u_0(x).
\end{cases}
\end{equation}
Here $\ka, m >0$; $0\leq a\in L^1(\X)\cap L^\infty(\R)$ with $\int_\X a(x)\,dx=1$; and $G$ is a nonnegative mapping on functions which is acting in $x$, i.e. $(Gu)(x,t):=\bigl(Gu(\cdot,t)\bigr)(x)\geq0$ for $u\geq0$.

We will distinguish two cases for the initial condition $u_0:\X\to\R_+$: 
\begin{gather}
\lim_{x\to\pm\infty} u_0(x)=0,\tag{C1}\label{eq:intcase}\\
\shortintertext{and} \lim_{x\to\infty} u_0(x)=0, \qquad \inf_{x\leq -\rho}u_0(x)>0 \ \text{for some } \rho\geq0.\tag{C2}\label{eq:moncase}
\end{gather}
We will not assume any symmetricity of either $a(x)$ or, in the case \eqref{eq:intcase}, $u_0(x)$; in particular, each of them may behave differently at $\infty$ and $-\infty$.

The function $u(x,t)$ may be interpreted as the local density of an evolving in time system of entities which reproduce themselves, compete, and die. The~reproduction appears according to the dispersion, which is realized via the fecundity rate $\ka$ and the density $a$ of a probability dispersion  distribution. The~death may appear due the constant inner mortality $m>0$ within the system, as well as due to the density dependent rate $Gu\geq0$, which describes a competition within the system. For another interpretation for the equation \eqref{eq:basicequation} rewritten in the reaction-diffusion form \eqref{eq:RDequation} and further references, see below and also \cite{FKT2016}. 

We consider \eqref{eq:basicequation} in the space $E:=L^\infty(\X)$ with the standard $\esssup$-norm.
By a solution to \eqref{eq:basicequation} on $\R_+$, we understand the so-called classical solution, that is a mapping $u:\R_+\to E$ which is continuous in $t\in\R_+$ and continuously differentiable (in the sense of the norm in $E$) in $t\in(0,\infty)$.

We start with the following assumptions:
\begin{assum}\label{assum:kappa>m}
        \beta:=\ka-m>0.
\end{assum}
\vspace{-1.35\baselineskip}
\begin{assum}\label{assum:Gpositive}
        \begin{gathered}
        \textsl{There exists $\theta>0$, such that, for each } 0\leq v\leq\theta,\\
                0=G0 \leq Gv\leq G\theta=\beta.
        \end{gathered} 
\end{assum} 
Here and below we write $v\leq w$ for $v,w\in E$, if $v(x) \leq w(x)$ for a.a.~$x\in \X$. Note that we will often just write $x\in\X$ omitting `for a.a.' before this.

As a result, $u\equiv0$ and $u\equiv\theta$ are stationary solutions to \eqref{eq:basicequation}. 
The rest of assumptions, \eqref{assum:Glipschitz}--\eqref{assum:approx_of_basic}, are considered in Section~\ref{sec:assum} below. In particular, they ensure that, between $0$ and $\theta$, there are not other constant stationary solutions to \eqref{eq:basicequation}; and also that $u\equiv0$ is an asymptotically unstable solution and $u\equiv\theta$ is an asymptotically stable one. Because of this, the equation \eqref{eq:basicequation} belongs to the class of the so-called monostable equations, see e.g. \cite{BH2002}. 
One can also rewrite \eqref{eq:basicequation} in the so-called reaction-diffusion form
\begin{equation}\label{eq:RDequation}
\begin{cases}
\dfrac{\partial}{\partial t} u(x,t)=\ka(a*u)(x,t)-\ka u(x,t)+(Fu)(x,t),\\[2mm]
u(x,0)=u_0(x),
\end{cases}
\end{equation}
where the symbol $*$ stands for the classical convolution on $\R$, i.e.
\begin{align}\label{eq:defconv}
(a*v)(x)&:=\int_\R a(x-y)v(y)\,dy, \quad x\in\X,
\\\shortintertext{and the reaction $F$ is given by}
\label{eq:FthroughG}
Fv&:=v(\beta -Gv), \quad v\in E.
\end{align}
Then, under assumptions \eqref{assum:kappa>m}--\eqref{assum:Gpositive}, we will have that
\begin{equation}\label{eq:propertiesofF}
F\theta=F0=0\leq Fv \leq \beta v, \quad 0\leq v\leq \theta.
\end{equation}
The assumption \eqref{assum:Glipschitz} below yields, in particular, that $G$ is continuous at $0$ on $\{v\in E: 0\leq v\leq\theta\}$; as a result, we require that the reaction $F$ in \eqref{eq:RDequation} is such that  $\frac{Fv}{v}\to\beta>0$ as $v\to0+$ (both convergences are in $E$). Because of $F0=0$, we get then that the Fr\'echet derivative of $F$ must be a (strictly positive) constant mapping. In particular, we do not allow the degenerate reaction $F'(0)=0$, see e.g. \cite{AC2016} and cf. Example~\ref{ex:local} below. Therefore, we consider a sub-class of monostable reaction-diffusion equations of the form \eqref{eq:RDequation}. 

The solution $u$ to the equation \eqref{eq:RDequation} may be interpreted as a density of a species which invades according to a nonlocal diffusion within the space $\X$ meeting a reaction $F$, see e.g. \cite{Fif1979,Mur2003}.
In the recent decade, there is a growing interest to the study of nonlocal monostable reaction-diffusion equations,
see e.g. \cite{PS2005,BCGR2014,CDM2008,Yag2009,Gar2011,CD2005,SLW2011}; for the origins of the topic see also \cite{Sch1980,Die1978a,Aro1977,Wei1978}. 

We will distinguish two main classes of the examples for $G$ or $F$, which fulfill the assumptions of Section~\ref{sec:assum}; see \cite[Examples 1.6--1.8]{FKT2016} for further details and references. 
Note that, in both examples of $F$ below, the mapping $Gu=\beta-\frac{Fu}{u}$ is well-defined, cf. \eqref{eq:FthroughG}.

\begin{example}[Reaction--diffusion equation with a local reaction]\label{ex:local}
Consider \eqref{eq:RDequation} with $F(u) = f(u)$ for a function $f:\R\to\R$ which satisfies the following assumptions, for some $\theta>0$,
\begin{equation*}
\begin{gathered}
        f \textsl{ is Lipschitz continuous on } [0,\theta]; \\
        \lim\limits_{r\to 0 +} \frac{f(r)}{r} = \beta;\\ 
        f(0)=f(\theta)=0;\quad 0<f(r) \leq \beta r,  \ r\in(0,\theta). 
\end{gathered}
\end{equation*}
In particular, if $f$ is differentiable at $0$, then we require $f'(0)=\beta>0$. 
\end{example}

\begin{example}[Spatial logistic equation and its generalizations]\label{ex:spataillogistic}
Consider a function $0\leq a^-\in L^1(\R)$ with $\int_\R a^-(x)\,dx=1$, such that, for some $\delta>0$,
\begin{equation}\label{eq:sepfromzerobeta}
        \ka a(x) - \beta a^-(x) \geq \delta\1_{B_\delta(0)}(x), \qquad x\in\X.
\end{equation}
Here and below, 
\[
B_r(x_0):=[x_0-r,x_0+r], \quad r>0, \ x_0\in\R.
\] 
Take an arbitrary $\theta>0$ and consider \eqref{eq:RDequation} with 
\begin{equation}\label{eq:ourexampleforF}
Fu=\gamma_k u(\theta - a^-*u)^k, \qquad \gamma_k:=\frac{\beta}{\theta^k}, \ k\in\N.
\end{equation}
\end{example}

To formulate our main result, we start with the following definition.
\begin{definition} \label{def:ltrt}
Let $\beta>0$ be given by \eqref{assum:kappa>m}. 
\begin{enumerate}
        \item Let $b:\R\to\R_+$ be continuous and strictly decreasing on $(\rho,\infty)$, for some $\rho>0$, with $\lim\limits_{x\to\infty} b(x)=0$. Then, for some $\tau>0$, there exists a function $r(t)=r(t,b)$, $t>\tau$, which
        uniquely solves the equation
\begin{equation}\label{eq:explicit}
b\bigl(r(t)\bigr)=e^{-\beta t}, \quad t>\tau,
\end{equation}
and $r(t)\to\infty$, $t\to\infty$. 
        \item Similarly, if the function $b$ is continuous and strictly increasing on $(-\infty,-\rho)$ with $\lim\limits_{x\to-\infty} b(x)=0$, then one can define $l(t)=l(t,b)\to\infty$, $t\to\infty$ as the unique solution to the equation
\begin{equation}\label{eq:explicit2}
b\bigl(-l(t)\bigr)=e^{-\beta t}, \quad t>\tau.
\end{equation}
\end{enumerate}
\end{definition}

In other words, $r(t)$ and $l(t)$ are given through the inverse functions to $-\log b$, namely, for $t>\tau$,
\begin{equation}\label{eq:coollogbinv}
r(t,b)=\bigl(-\log b\upharpoonright_{\R_+}\bigr)^{-1}(\beta t), \qquad l(t,b)=\bigl(-\log b\upharpoonright_{\R_-}\bigr)^{-1}(\beta t). 
\end{equation}

We are going to find sufficient conditions on $a$ and $u_0$, such that the corresponding solution $u$ to \eqref{eq:basicequation}, in the case  \eqref{eq:intcase}, becomes arbitrary close to $\theta$ (as $t$ goes to $\infty$) inside the (expanded) interval $(-l(t),r(t))$ and becomes arbitrary close to $0$ outside of this interval. In the case \eqref{eq:moncase}, one has to consider 
the interval $(-\infty,r(t))$ instead. Here $l(t)=l(t,b)$ and $r(t)=r(t,b)$, where, cf.~\eqref{eq:coollogbinv},
\begin{align}
&\log b(x)\sim \log \max\bigl\{ a(x), \ u_0(x)\bigr\}, \ x\to\infty, && \text{if \eqref{eq:intcase} holds, }\label{eq:defintc}\\[3mm]
&\log b(x)\sim \log \max\biggl\{ \int_x^\infty a(y)\,dy, \ u_0(x)\biggr\}, \ x\to\infty,  && \text{if \eqref{eq:moncase} holds,}\label{eq:defmonc}
\end{align}
and we suppose that the function $b$ has regularly heavy tails at $\infty$, see Definition~\ref{def:reght} below. Here and below the notation $f(x)\sim g(x)$, $x\to\infty$ means that $\frac{f(x)}{g(x)}\to1$, $x\to\infty$. In~particular, for any small $\eps,\delta>0$, we will have that
\[
\bigl\{x>0\bigm\vert u(x,t)\in(\delta,\theta- \delta)\bigr\}\subset \bigl(r(t-t\eps), r(t+t\eps)\bigr)
\]
for all $t$ big enough; in the case \eqref{eq:intcase}, the corresponding result also holds for negative values of $x$ and the function $l(t)$ instead.

\begin{definition}\label{def:reght}
\begin{enumerate}
        \item A bounded function $b:\R\to\R_+$ is said to have a regularly heavy tail at $\infty$ in the sense of densities, if
        $b\in L^1(\R_+)$, $b$ is decreasing to $0$ and convex on $(\rho,\infty)$ for some $\rho>0$, and
\begin{gather}b\bigl(x+y\bigr)\sim b(x), \quad y\in\R,\ x\to\infty, \label{eq:long-tailed_dens}\\ 
\int_0^x b(x-y) b(y) \,dy\sim 2 \biggl(\int_{\R_+}b(y)\,d y\biggr) b(x) , \quad x\to\infty.\label{eq:sub-exp_dens}
\end{gather}
A bounded function $b:\R\to\R_+$ is said to have a regularly heavy tail at $-\infty$ in the sense of densities, if the function $b(-x)$ has a regularly heavy tail at $\infty$ in the sense of densities.
\item A bounded function $b:\R\to\R_+$ is said to have a regularly heavy tail at $\infty$ in the sense of distributions, if $b$ is decreasing to $0$ on $\R$, $b$ is convex on  $(\rho,\infty)$ for some $\rho>0$, and
\begin{equation}
-\int_0^x b(x-y)\, d b(y) \sim 2 b(-\infty) b(x) , \quad x\to\infty,\label{eq:sub-exp_distr}
\end{equation}
where $db(y)$ is the Lebesgue--Stieltjes measure associated with $b$.
\end{enumerate}
\end{definition}

\begin{remark}\label{rem:longtaileddistr}
By \cite[Lemmas 3.2, 3.4 and Definition~2.21]{FKZ2013}, \eqref{eq:sub-exp_distr} implies \eqref{eq:long-tailed_dens}.
\end{remark}

\begin{remark}
Note that if $b:\R\to\R_+$ has a regularly heavy tail at $\infty$ in the sense of densities and $b\in L^1(\R)$, then the function
\begin{equation}\label{eq:primeb}
B(x):=\int_x^\infty b(y)\,dy, \quad x\in\R
\end{equation}
 has a regularly heavy tail at $\infty$ in the sense of distribution. The inverse statement is not, in general, true, cf. \cite[Section 4.2]{FKZ2013}.
\end{remark}

Examples of functions with regularly heavy tails at $\infty$  in the sense of densities are the following: 
\begin{equation}\label{eq:examplesofS}
\begin{aligned}
&(\log x)^\mu x^{-q}, &&  \qquad  && (\log x)^\mu  x^\nu \exp\bigl(-p(\log x)^q\bigr), \\
&(\log x)^\mu x^\nu \exp\bigl(-x^\alpha\bigr),  && \qquad &&   
(\log x)^\mu x^\nu \exp\Bigl(-\frac{x}{(\log x)^q}\Bigr),
\end{aligned}
\end{equation}
where $p>0$, $q>1$, $\alpha\in(0,1)$, $\nu,\mu\in\R$. See also Lemma~\ref{le:basicforb} below for a sufficient condition, which can be checked for further `intermediate' asymptotics at $\infty$.  
To get examples of functions with regularly heavy tails at $\infty$  in the sense of distributions, one can use \eqref{eq:primeb}. 

Note that, see Lemma~\ref{le:basicforb} for details, any $b$ with a regularly heavy tail at $\infty$ in the sense of densities  is such that, for each $k>0$,
\[
e^{kx}b(x)\to\infty, \quad x\to\infty;
\] 
this explains the name: the tail of $b$ at $\infty$ is `heavier' than the tail of an exponential function. By~Remark~\ref{rem:longtaileddistr}, the same property has each $b$ with a regularly heavy tail at $\infty$ in the sense of distributions. 

Now one can formulate our main result; recall that the exact formulations for the assumptions \eqref{assum:Glipschitz}--\eqref{assum:approx_of_basic} are given in Section~\ref{sec:assum} below.
\begin{theorem}\label{thm:fullmain}
Let either \eqref{assum:kappa>m}--\eqref{assum:improved_sufficient_for_comparison} hold or \eqref{assum:approx_of_basic} hold.  Let $0\leq u_0\leq\theta$, $u_0\not\equiv0$ and $u$ be the corresponding solution to \eqref{eq:basicequation}. 
\begin{enumerate}
        \item Let $u_0$ satisfy \eqref{eq:intcase} and functions $b,b_1,b_2:\R\to\R_+$ have regularly heavy tails at both $\pm\infty$ in the sense of densities, and the following assumptions hold
        \begin{gather}
        \text{either} \quad u_0(x)\geq b_1(x) \quad \text{or} \quad  a(x)\geq b_1(x), \qquad x\in\R, \label{eq:eqandbdd}\\
        \max\bigl\{ a(x), \ u_0(x)\bigr\}\leq b_2(x), \quad x\in\R.\label{eq:eqandbdd2}
         \end{gather}
         Suppose also that
         \begin{equation}
         \log b_1(x)\sim\log b_2(x)\sim\log b(x), \label{eq:eqandbdd3}
         \end{equation}
         as $x\to\pm\infty$. Then, for each $\eps\in(0,1)$, 
        \begin{align}
                &\lim_{t\to\infty}\essinf_{[-l(t-\eps t,b), r(t-\eps t,b)]}u(x,t)=\theta,\label{eq:main1}\\[2mm]
                &\lim_{t\to\infty}\esssup_{(-\infty,-l(t+\eps t,b)]\cup [r(t+\eps t,b),\infty)}u(x,t)=0.\label{eq:main2}
        \end{align}
        \item Let $u_0$ satisfy \eqref{eq:moncase} and functions $b,b_1,b_2:\R\to\R_+$ have regularly heavy tails at $\infty$ in the sense of distributions, and the following assumptions hold
        \begin{gather}
        \text{either} \quad u_0(x)\geq b_1(x) \quad \text{or} \quad  \int_x^\infty a(y)\,dy\geq b_1(x), \qquad x\in\R, \label{eq:eqandbdd4}\\
        \max\biggl\{ \int_x^\infty a(y)\,dy, \ u_0(x)\biggr\}\leq b_2(x), \quad x\in\R.\label{eq:eqandbdd5}
        \end{gather}
        Suppose also that \eqref{eq:eqandbdd3} holds as $x\to\infty$. Then, for each $\eps\in(0,1)$, 
                \begin{align}
                &\lim_{t\to\infty}\essinf_{(-\infty, r(t-\eps t,b)]}u(x,t)=\theta,\label{eq:main3}\\[2mm]
                &\lim_{t\to\infty}\esssup_{[r(t+\eps t,b),\infty)}u(x,t)=0.\label{eq:main4}
                \end{align}
\end{enumerate}
\end{theorem}
\begin{remark}\label{rem:agreementnonzero}
For a brevity of notations, we treat here and in the sequel the condition $u_0\not\equiv0$ as follows: there exist $\delta,\rho>0$ and $x_0\in\X$, such that $u_0(x)\geq \delta$ for a.a.~$x\in B_\rho(x_0)$. 
\end{remark}
\begin{remark}
We will see in Theorems~\ref{thm:bb:est_below} and \ref{thm:bb:est_above} below, that the assumptions on $b_1$ and $b_2$ may be slightly weaken.
\end{remark}

Stress that the convergences in \eqref{eq:main1}--\eqref{eq:main2} or \eqref{eq:main3}--\eqref{eq:main4} are indeed `accelerated' in $t$, since, because of \eqref{eq:fastewwqr}, each $b:\R\to\R_+$ with regularly heavy tail(s) (in either of senses) satisfies, for each $k>0$,
\begin{equation}\label{eq:acceleff}
r(t,b)-k t\to\infty, \qquad  l(t,b)-k t\to\infty, \qquad t\to\infty.
\end{equation}

The reason to introduce the function $b$ in Theorem~\ref{thm:fullmain} is two-fold. First, we allow some flexibility in the choice of $b_1$ and $b_2$ and hence of $a$ and $u_0$. For example, $b_1$ may be a function from \eqref{eq:examplesofS} with negative values of $\nu$ and $\mu$, whereas $b_2$ may be `the same' function, but with positive values of $\nu$ and $\mu$; then $u_0$ and $a$ (or $\int_x^\infty a$ in the second part of Theorem~\ref{thm:fullmain}) may fluctuate between such $b_1$ and $b_2$. In this case, one can take $b$ equal to `the same' function, but with $\nu=\mu=0$, since then \eqref{eq:eqandbdd3} evidently holds. Secondly, choosing such $b$, one can find $r(t,b)$ explicitly (i.e. \eqref{eq:explicit} can be solved). Namely, cf.~\cite[Example~2.18]{FKT2016}, one has the following values of $r(t)=r(t,b)$:
\begin{equation*}
\begin{aligned}
&b(x)=x^{-q}, &&  \qquad && r(t)= \exp\Bigl(\dfrac{\beta }{q}t\Bigr);\\
& b(x)=\exp\bigl(-p(\log x)^q\bigr),  \qquad && 
&& r(t)=\exp\biggl(\Bigl(\frac{\beta}{p}t\Bigr)^{\frac{1}{q}}\biggr);\\
&b(x)=\exp\bigl(-x^\alpha\bigr),  &&  \qquad && r(t)=(\beta t)^{\frac{1}{\alpha}} ;\\   
&b(x)=\exp\Bigl(-\frac{x}{(\log x)^q}\Bigr), &&  \qquad && r(t)\sim \beta t (\log t)^q, t\to\infty.
\end{aligned}
\end{equation*}
(Recall that here $q>1$, $\alpha\in(0,1)$, $p>0$.) Note that, for $b(-x)$ and $l(t)=l(t,b)$, the same examples hold.

\begin{remark}\label{rem:differentspeeds}
 In view of \eqref{eq:coollogbinv}--\eqref{eq:defmonc}, the asymptotic of $r(t)$ may be different in the cases \eqref{eq:intcase} and \eqref{eq:moncase} for the same kernel $a$. For example, let $a(x)=x^{-q}$, $q>2$, for large $x$; then $\int_x^\infty a(y)\,dy$ is proportional to $x^{-q+1}$ for large $x$. Therefore, if $u_0$ decays at $+\infty$ faster than $x^{-q}$, then, in the case \eqref{eq:intcase}, we will get $r(t)= \exp\bigl(\frac{\beta }{q}t\bigr)$, whereas, in the case \eqref{eq:moncase}, we will get $r(t)= \exp\bigl(\frac{\beta }{q-1}t\bigr)$. Hence the propagation in the latter case will be faster. 
 \end{remark}

Our method is based on the usage of functions with regularly heavy tails because of the following reasons. 
The conditions on $G$ we require imply the effect called the linear determinacy in e.g. \cite{Wei2012}, or linear selection in \cite{LMN2004}, cf. also the pulled fronts in \cite{GGHR2012}. The effect is that the long-time behavior of the solutions to \eqref{eq:basicequation} is well-described by the solutions to the corresponding problem linearized at the unstable stationary solution $u\equiv0$, that is
\begin{equation}\label{eq:linearequation}
\begin{cases}
\dfrac{\partial}{\partial t} w(x,t) = \ka \displaystyle \int_\X a(x-y)w(y,t)\,dx-m w(x,t), \\[3mm]
w(x,0)=u_0(x).
\end{cases}
\end{equation} 
Indeed, note that the term $u(x,t) (Gu)(x,t)$ in  \eqref{eq:basicequation} is small for `big' values of $x$ relatively to $u(x,t)$, provided that $G$ is continuous at $0\in E$. Next, because of \eqref{assum:Gpositive}, we have that $u(x,t)\leq w(x,t)$ for all $x$ and $t$. The solution to \eqref{eq:linearequation} is given through a series of the convolution powers, and the main peculiarity of the functions with regularly heavy tails at infinity is that their convolution powers can be estimated by the functions themselves. For the monotone case related to \eqref{eq:moncase}, it was the classical Kesten's bound for distributions on $\R$, see e.g. \cite[Theorem 3.34]{FKZ2013} and Lemma~\ref{lem:eqqe} below. For the integrable case related to \eqref{eq:intcase}, we used our extension of Kesten's bound to the densities on $\R$, see \cite[Theorem 2.22]{FT2017b} and Lemma~\ref{lem:eqqeour} below. This is the main tool to get \eqref{eq:main2} and \eqref{eq:main4}, see Theorem~\ref{thm:bb:est_above}. 

In order to prove the convergence to $\theta$ as well, namely, to get \eqref{eq:main1} and \eqref{eq:main3}, 
we construct in \eqref{eq:subsol_to_lin_eq_delta} a minorant $g(x,t)$ to the solution $u(x,t)$ to \eqref{eq:basicequation}, which is a sub-solution to the linear equation \eqref{eq:linearequation} with $m$ replaced by $m+\delta$ for a small $\delta>0$. The detailed realization of the lower estimates is given in Section~\ref{sec:convtotheta}.

For an overview of the existing results about the propagation of solutions to~\eqref{eq:basicequation} (even over $\R^d$, $d\geq1$), we refer the reader to \cite[Subsection~1.5]{FKT2016}. In brief, for the case $d=1$ considered in the present paper, the  situation is the following. If both the kernel $a$ and the  initial condition $u_0$ are light-tailed, more precisely, if $a$ is exponentially integrable and $u_0$ is exponentially bounded, then, for example, \eqref{eq:main1}--\eqref{eq:main2} hold for linear $r(t)=c_+ t$ and $l(t)=c_- t$ (with explicit formulas for $c_\pm\in\R$). This case corresponds to the (non-accelerated) linear dispersion spreading.

The accelerated case for the local non-linearity (see Example~\ref{ex:local}) was known in the mathematical biology, see e.g. \cite{MK2003}. The first rigorous result in this direction was done by Garnier~\cite{Gar2011}, who proved an analogue of  \eqref{eq:main1}--\eqref{eq:main2} for a compactly supported initial condition $u_0$ and symmetric heavy-tailed kernel $a$, such that \eqref{assum:first_moment_finite} holds. However, in his approach, instead of the function $r(t+\eps t,b)$
in \eqref{eq:main2} with arbitrary small $\eps>0$, appeared this function with an unknown $\eps>0$, i.e. the result was not sharp. 

Our results in \cite{FKT2016}, being rephrased for the case $d=1$, yield both \eqref{eq:main1}--\eqref{eq:main2} for \eqref{eq:intcase} and \eqref{eq:main3}--\eqref{eq:main4} for \eqref{eq:moncase}, provided that the function $b$ in Theorem~\ref{thm:fullmain} was symmetric for \eqref{eq:intcase} (and  $r(t)=l(t)$ then), and $b$ was the antiderivative, cf. \eqref{eq:primeb}, of a symmetric function for \eqref{eq:moncase}. Note that the functions $a$ and $u_0$ were not need to be symmetric, up to the equivalence \eqref{eq:eqandbdd3} though. In~particular, either $a$ or $u_0$ (but not both) might be still light-tailed. Recently, an alternative approach was proposed in \cite{BGHP2017}, where an analogous result to  \eqref{eq:main1}--\eqref{eq:main2} was obtained (for $d=1$ and in the special case of Example~\ref{ex:local}), provided that both the kernel $a$ and the initial condition $u_0$ are symmetric, and $a$ is heavy-tailed; the technique used there goes back to \cite{ES1989}.  In another recent paper \cite{AC2016}, also for the case of Example~\ref{ex:local}, a similar result was obtained for a non-necessary symmetric  $a$ and for  $u_0$ which satisfies \eqref{eq:moncase} with the additional restriction that $u_0(x)=0$ for large $x$. 

Therefore, the present paper is seemed to be the first one which deals with the case when either of the kernel and the initial condition has (perhaps different) heavy tails at both $\pm\infty$ in the case of an integrable initial condition, or considers a monotone-like initial condition  which is not necessarily vanishing at $+\infty$. We stress that Theorem~\ref{thm:fullmain} shows that the acceleration for the propagation of the solution to \eqref{eq:basicequation} to the right direction is fully determined by the right tails of either $a$ or $u_0$.

Note also the effect similar to the observed in Remark~\ref{rem:differentspeeds} about the possibility of different speeds for the cases \eqref{eq:intcase} and \eqref{eq:moncase}  was also shown for an analogue of \eqref{eq:basicequation} with the fractional Laplacian (in particular, when $a$ is singular and non-integrable), see \cite{CR2013,FY2013}.

\section{Assumptions and properties}\label{sec:assum}
Describe now the rest of our assumptions. The first ones guarantee the existence-uniqueness and comparison results of Proposition~\ref{thm:existandcompared} below. 

\vspace{-0.95\baselineskip}
\begin{assum}\label{assum:Glipschitz}
        \begin{gathered}
                \textsl{There exists $\lt>0$, such that, for each }0\leq v,w \leq \theta\\
                \|Gv-Gw\| \leq \lt \|v-w\|.
        \end{gathered}
\end{assum}
\vspace{-0.95\baselineskip}
\begin{assum}\label{assum:sufficient_for_comparison}
        \begin{gathered}
                \textsl{For some $p\geq0$ and for each $0\leq v\leq w\leq\theta$},\\
        \ka a*v -v\, Gv + pv \leq \ka a*w -w\, Gw + pw.
        \end{gathered}
\end{assum}

\begin{proposition}[\!\!\!{\cite[Theorems 2.1, 2.2, Proposition 4.2]{FT2017a}}]\label{thm:existandcompared}
Let assumptions \eqref{assum:kappa>m}--\eqref{assum:sufficient_for_comparison} hold, and $0\leq u_0\leq \theta$. Then, for each $T>0$, there exists a unique solution $u=u(x,t)$ to \eqref{eq:basicequation} for $t\in[0,T]$; and 
\begin{equation}\label{eq:mainprop}
0\leq u(\cdot,t)\leq\theta \quad \text{for all } t>0.
\end{equation}
 Moreover, let $0\leq v_0\leq\theta$ and $v=v(x,t)$ be the corresponding solution to $\eqref{eq:basicequation}$; then $u_0\leq v_0$ implies that 
 \begin{equation}\label{eq:compofsol}
 0\leq u(\cdot,t)\leq v(\cdot,t)\leq \theta\quad \text{for all } t>0.
 \end{equation}
\end{proposition}

\begin{remark}
Note that the assumption \eqref{assum:kappa>m}  excludes the trivial case when $u(x,t)$ converges to $0$ as $t\to\infty$ uniformly in $x\in\R$. Next, 
we have shown in \cite[Theorems 2.2]{FT2017a} that the assumptions  \eqref{assum:kappa>m}--\eqref{assum:sufficient_for_comparison}
are sufficient to get the comparison principle for solutions to \eqref{eq:basicequation} with initial conditions $0\leq u_0 \leq \theta$ (for the exact formulation see also Lemma~\ref{lem:comparison} below). 
For particular cases of $G$ the assumption \eqref{assum:sufficient_for_comparison} is also a necessary condition for the comparison. 
For instance, in the case of Example~\ref{ex:spataillogistic} with $k=1$, the condition \eqref{assum:sufficient_for_comparison}  reads as
\[
        \ka a(x) \geq (\ka-m) a^-(x), \quad x\in\X,
\]
cf.  \eqref{eq:sepfromzerobeta}. It was shown in \cite[Remark 3.6]{FKT2015}, that if the latter inequality fails, the solution may not satisfy \eqref{eq:mainprop}.
\end{remark}

The rest of assumptions we need, in particular, to show that a solution $u=u(x,t)$ to \eqref{eq:basicequation} converges to $\theta$ locally in space, when time tends to $\infty$, see Proposition~\ref{thm:hairtrigger} below for the exact formulation.
\begin{assum}\label{assum:a_nodeg}
        \textsl{There exist $\rho, \delta>0$, such that} \ a(x)\geq\rho \text{ for } |x|\leq\delta.
\end{assum}
\vspace{-0.95\baselineskip}
\begin{assum}\label{assum:G_locally_continuous}
        \begin{gathered}        
                \textsl{For any $0\leq v_n, v\leq \theta$, such that $v_n \locun v$, $n \to \infty$,}\\
                Gv_n \locun Gv, \ n \to \infty,
        \end{gathered}
\end{assum}
\vspace{-0.95\baselineskip}%
where $\locun$ means uniform convergence on all compact subsets of $\X$. 

\begin{assum}\label{assum:G_commute_T}
        \begin{gathered}
        \textsl{For each $y\in\X$ and $0\leq v\leq\theta$,}\\
                (T_y G v)(x) = (G T_y v)(x)\qquad \text{for }  x\in\X, 
        \end{gathered}
\end{assum}
where $T_y:E\to E$, $y\in\X$ is the translation operator, given by
\begin{equation}\label{shiftoper}
        (T_y v)(x):=v(x-y), \quad x\in\X.
\end{equation}

The condition \eqref{assum:G_commute_T} implies that, for any $r\equiv const\in (0,\theta)$, $Gr\equiv const$. We~will assume then also that
\begin{assum}\label{assum:G_increas_on_const}
  Gr < \beta, \qquad r\in(0,\theta).    
\end{assum}

Finally, we will distinguish two cases. If the condition
\begin{assum}\label{assum:first_moment_finite}
        \int_\X \lvert y\rvert a(y) d y <\infty
\end{assum}
holds, then we set
\begin{equation}\label{firstfullmoment}
  \ma:=\ka \int_\X y a(y)\,dy,
\end{equation}
and assume, additionally to \eqref{assum:sufficient_for_comparison}, that
\begin{assum}\label{assum:improved_sufficient_for_comparison}
        \begin{gathered}
                \textsl{there exist $p\geq 0$, $0\leq b\in C^\infty(\X)\cap L^\infty(\X)$, $\delta>0$, such that}\\ 
                a(x)-b(x)\geq \delta\1_{B_\delta(0)}(x), \quad x\in\X,\\
        w\, Gw\leq \ka b*w + pw \qquad \text{for } 0\leq w\leq\theta.
        \end{gathered}
\end{assum}

Otherwise, if \eqref{assum:first_moment_finite} does not hold, then we assume that,
\begin{assum}\label{assum:approx_of_basic}
        \begin{gathered}
                \textsl{for each $n\in\N$, there exist} \\
                0\leq a_n\in L^1(\X), \quad \ka_n>0, \quad G_n:E\to E, \quad \theta_n\in(0,\theta]\\ 
                \textsl{which satisfy \eqref{assum:kappa>m}--\eqref{assum:improved_sufficient_for_comparison} instead of $a$, $\ka$, $G$, $\theta$, }\\ \textsl{correspondingly, such that}\\
                \ma_n:=\ka_n\int_{\X} ya_n(y)dy\in\R, \qquad \theta_n>\theta- \frac{1}{n}, \qquad n\in\N,  \\ 
                \ka_n a_n*w -wG_nw \leq \ka a*w - wGw\qquad \text{for } 0\leq w\leq \theta_n, \ n\in\N.
        \end{gathered}
\end{assum}

If \eqref{assum:first_moment_finite} does not hold (e.g. $a(x)\sim |x|^{-2}$ as $x\to\infty$ and/or $x\to-\infty$, see the main results below), then, to fulfill \eqref{assum:approx_of_basic} in Examples~\ref{ex:local}--\ref{ex:spataillogistic}, we choose, for $m\in(0,\ka)$, a sequence of sets $\La_n\subset \X$, $\La_n\nearrow \X$, such that $\ka_n: = \ka \int_{\La_n} a(x) dx>m$, and define $a_n(x):= \bigl(\int_{\La_n} a(x) dx\bigr)^{-1}\1_{\La_n}(x) a(x)$, $x\in\X$. In Example~\ref{ex:local}, we take $G_n:=G$, whereas, in Example~\ref{ex:spataillogistic} and $k=1$ in \eqref{eq:ourexampleforF} (the general $k$ can be considered analogously), we set $G_n u:=\ka^-a^-_n*u$, where $a_n^-(x):=\1_{\La_n}(x)a^-(x)$, $x\in\X$. Since $\theta_n:=(\ka_n-m)/\bigl(\ka^- \int_\X a_n^-(x)\,dx\bigr)\to\theta$, $n\to\infty$, one can assume that $\theta_n>\theta-\frac{1}{n}$. 

\begin{proposition}\label{thm:hairtrigger}
Let $0\leq u_0\leq\theta$, $u_0\not\equiv0$, and let $u=u(x,t)$ be the corresponding solution to~\eqref{eq:basicequation}.
\begin{enumerate}
        \item (\!\!\!{\cite[Theorem 2.3]{FT2017a}}) Let \eqref{assum:kappa>m}--\eqref{assum:improved_sufficient_for_comparison} hold and $\ma$ be given by \eqref{firstfullmoment}.
        Then, for each compact set $K\subset \X$,
  \begin{equation}\label{eq:hair-trig}
    \lim_{t\to\infty} \essinf_{x\in K} u(x+t\ma,t)=\theta.
  \end{equation}
  \item (cf. \cite[Theorem 2.5]{FT2017a}) Let \eqref{assum:approx_of_basic} hold.
  Then, for each compact set $K\subset \X$ and for each $n\in\N$,
  \begin{equation}\label{eq:hair-trig-mod}
    \liminf_{t\to\infty} \essinf_{x\in K} u(x+t\ma_n,t)\geq\theta- \frac{1}{n}.
  \end{equation}
\end{enumerate}
\end{proposition}

\begin{remark}
Note that the proof of the second statement in Proposition~\ref{thm:hairtrigger} is a straightforward modification of that in \cite[Theorem~2.5]{FT2017a}.
\end{remark}
\begin{remark}
Note also that we required in Definition~\ref{def:reght} an additional convexity at infinity of a regularly heavy-tailed function just to cover the case $\ma\neq0$ in \eqref{firstfullmoment} or $\ma_n\neq0$ in \eqref{assum:approx_of_basic}, see the usage of Proposition~\ref{eq:crucialforlinearshift} in the proof of Theorem~\ref{thm:bb:est_below} below.
\end{remark}

\section{Technical tools}\label{sec:classesoffunctions}
\begin{definition}\label{def:right-sideclasses}
 A function $b:\R\to\R_+$ is said to be 
 \begin{description}
        \item[--] {\em (right-side) long-tailed}, if there exists $\rho\geq0$, such that $b(x)>0$ for all $x\geq\rho$; and,~for~any $y\geq 0$,
                \begin{equation}\label{eq:longtaileddef}
                        \lim_{x\to\infty}\frac{b(x+y)}{b(x)}=1;
                \end{equation}
   \item[--] {\em (right-side) tail-decreasing (tail-continuous, tail-convex, tail-log-convex)}, if\linebreak $b(x)>0$, $x\in(\rho,\infty)$, for some $\rho\geq0$, and $b$ is strictly decreasing to $0$ (respectively, $b$ is continuous, $b$ is convex, $\log b$ is convex) on $(\rho,\infty)$;
   \item[--] {\em sub-exponential on $\R_+$} in the sense of densities, if $b\in L^1(\R_+)\cap L^\infty(\R_+)$, $b$ is long-tailed, and
        \begin{equation}
        \int_0^x b(x-y) b(y) \,dy\sim 2 \biggl(\int_{\R_+}b(y)\,d y\biggr) b(x) , \quad x\to\infty;
    \label{eq:defofsubexponR+}
        \end{equation}
   \item[--] {\em sub-exponential on $\R_+$} in the sense of distributions, if $b\in L^\infty(\R_+)$, $b$ is decreasing to $0$ on $\R_+$, and
     \begin{equation}
        -\int_0^x b(x-y)  \,db(y)\sim 2 b(-\infty )b(x) , \quad x\to\infty,
    \label{eq:defofsubexponR+distr}
        \end{equation}
        where $db(y)$ is the Lebesgue--Stieltjes measure associated with $b$.
 \end{description}
\end{definition}

\begin{remark}
If $b$ is sub-exponential on $\R_+$ in the sense of densities, then the function $\bigl(\|b\|_{L^1(\R_+)}\bigr)^{-1} b(x)$ is a sub-exponential probability density on $\R_+$, cf. e.g. \cite[Definition~4.6]{FKZ2013}.

If $b$ is sub-exponential on $\R_+$ in the sense of distributions, the function $\1_{\R_+}(x)\Bigl(1-\bigl( b(-\infty)\bigr)^{-1} b(x)\Bigr)$ is a sub-exponential probability distribution on $\R_+$, cf.~e.g.~\cite[Definition~3.1]{FKZ2013}.

As it was point out in Remark~\ref{rem:longtaileddistr}, cf. \cite[Lemmas 3.2, 3.4 and Definition~2.21]{FKZ2013}, if $b$ is sub-exponential in the sense of distributions, then \eqref{eq:longtaileddef} holds.
\end{remark}

\begin{lemma}\label{le:basicforb} Let $b:\R\to\R_+$ be (right-side) long-tailed. 
\begin{enumerate}
        \item (\!\!\cite[Lemma 2.17]{FKZ2013}) For each $k>0$,
                \begin{equation}\label{eq:subexplt}
                        \lim_{x\to\infty} e^{kx}b(x)=\infty;
                \end{equation}
        \item (\!\!\cite[Lemma 2.19, Proposition 2.20]{FKZ2013}) There exists a non-decreasing function $h:(0,\infty)\to(0,\infty)$, with
$h(x)<\dfrac{x}{2}$ and $\lim\limits_{x\to\infty}h(x)=\infty$, such that
\begin{equation}\label{eq:uniformlongtailedR}
  \lim_{x\to\infty}\sup_{|y|\leq h(x)}\biggl\lvert\frac{b(x+y)}{b(x)}-1\biggr\rvert= 0;
\end{equation}
        \item (\!\!\cite[Theorem~4.15, Section 4.2]{FKZ2013}) Let, additionally, $b$ be a tail-log-convex function, such that $b\in L^1(\R_+)$; and suppose that the function $h$ in \eqref{eq:uniformlongtailedR} can be chosen such that
\begin{equation}\label{eq:S0}
  \lim_{x\to\infty}x\, b\bigl(h(x)\bigr)=0.
\end{equation}
Then $b$ is sub-exponential on~$\R_+$ in the sense of densities and the function \eqref{eq:primeb}
is sub-exponential on $\R_+$ in the sense of distributions.
\end{enumerate}
\end{lemma}

\begin{remark}
By \cite[Proposition 2.15]{FT2017b}, if $b:\R\to\R_+$ is tail-decreasing and $0<h(x)<\frac{x}{2}$ with $\lim\limits_{x\to\infty}h(x)=\infty$, then \eqref{eq:uniformlongtailedR} is equivalent to
\begin{equation*}
\lim_{x\to\infty}\frac{b(x\pm h(x))}{b(x)}=1.
\end{equation*}
\end{remark}

\begin{example}[\!\!\!\protect{\cite[Subsection 3.2]{FT2017b}}]\label{ex:SR}
Let $b:\R\to\R_+$ be a bounded tail-decreasing and tail-log-convex function, such that, for some $C>0$, the function $Cb(x)$ has either of the asymptotics \eqref{eq:examplesofS} as $x\to\infty$, where $p>0$, $q>1$, $\alpha\in(0,1)$, $\nu,\mu\in\R$. Then $h(x)$ in \eqref{eq:uniformlongtailedR} can be chosen such that \eqref{eq:S0} holds; in particular, then $b$ is sub-exponential on $\R_+$. Note also that the functions \eqref{eq:examplesofS} themselves are tail-decreasing and tail-log-convex.
\end{example}

\begin{lemma}\label{lem:eqqeour}
Let   $b\in L^1(\R,\R_+)$   be sub-exponential on $\R_+$ in the sense of densities. Suppose that there exist $\rho, K>0$, such that
\begin{equation}\label{eq:quasi-decreasing}
  b(x+y)\leq K b(x), \quad x>\rho, \ y>0
\end{equation}
(for example, let $b$ be tail-decreasing). 
\begin{enumerate}
        \item (\!\!\cite[Theorem 2.19]{FT2017b}; for $n=2$, see also \cite[Lemma~4.13]{FKZ2013}) For each $n\geq2$, 
        \begin{equation*}
    \lim_{x\to\infty} \frac{b^{*n}(x)}{b(x)}=n\Bigl(\int_\R b(y)\,dy\Bigr)^{n-1}, 
  \end{equation*}
  where $b^{*n}(x)=(\underbrace{b*\ldots * b}_n)(x)$, $x\in\R$, and $*$ is given by \eqref{eq:defconv}.
  \item (\!\!\!\cite[Theorem 2.22]{FT2017b}) Let, additionally, $b$ be bounded and there exist a bounded $d:\R\to\R_+$ which is sub-exponential on $\R_+$, such that \eqref{eq:quasi-decreasing} holds with $b$ replaced by $d$, and, for some $D>0$ and $\rho'>0$,
  \begin{equation}\label{eq:dsaadsewr23}
  b(-x)\leq D \, d(x), \quad  x>\rho'
  \end{equation}
  (for example, let \eqref{eq:dsaadsewr23} hold with $d(x)=\frac{1}{1+|x|^{1+\delta}}$, $x\in\X$, $\delta>0$).
  Then, for any $\delta\in(0,1)$, there exist $C_\delta, x_\delta>0$, such that 
\begin{equation}\label{eq:Kb1}
        b^{*n}(x)\leq C_\delta (1+\delta)^{n}\Bigl(\int_\R b(y)\,dy\Bigr)^{n-1}  b(x)
\end{equation}
for all $x>x_\delta$, $n\in\N$.
\end{enumerate}
\end{lemma}

\begin{lemma}\label{lem:eqqe}
Let $b:\R\to\R_+$ be bounded, continuous, and decreasing on $\R$ to $0$ function, which is sub-exponential on $\R_+$ in the sense of distributions. Denote $b^{\star 1}(x):=b(x)$, $x\in\R$. For each $n\geq2$, we consider the Lebesgue--Stieltjes integral
\begin{equation}\label{eq:convofdistr}
b^{\star n}(x):=-\int_\R b(x-y)d b^{\star (n-1)}(y), \quad x\in\R.
\end{equation}
\begin{enumerate}
        \item (\!\!\cite[Corollary 3.20]{FKZ2013}) For each $n\geq 2$,
        \begin{equation*}
                b^{\star n}(x)\sim n \bigl(b(-\infty)\bigr)^{n-1} b(x), \quad x\to\infty.
        \end{equation*}
        \item (\!\!\cite[Theorem 3.34]{FKZ2013}) For each $\delta\in(0,1)$, there exists $C_\delta>0$, such that
        \begin{equation}\label{eq:Kb1distr}
                b^{\star n}(x)\leq C_\delta (1+\delta)^n \bigl(b(-\infty)\bigr)^{n-1} b(x), \quad x\geq0.
        \end{equation} 
\end{enumerate}
\end{lemma}

\begin{remark}
Let $b\in L^1(\R)$ and $B$ be given by \eqref{eq:primeb}. Then
\begin{equation*}
B^{\star n}(x)=\int_x^\infty b^{*n}(y)\,dy, \quad x\in\R.
\end{equation*}
Recall that here by $\star$ we denote the convolution \eqref{eq:convofdistr} of decreasing bounded functions on the real line (e.g. tails of probability distributions), whereas by $*$ we denote the convolution \eqref{eq:defconv} of integrable functions on the real line (e.g. probability densities).
\end{remark}

\begin{lemma}[\!\!\protect{\cite[Lemma 2.15]{FKT2016}}]\label{le:nonconstantspeed}
  Let $b:\R \to \R_+$ be (right-side) tail-decreasing and long-tailed function. Then, for any $k>0$,
 \begin{equation}\label{eq:fastewwqr}
 r(t,b)-kt\to\infty, \quad t\to\infty.
 \end{equation} 
\end{lemma}
Evidently, if $b(-x)$ is (right-side) tail-decreasing and long-tailed, \eqref{eq:fastewwqr} holds for $l(t,b)$.

For Theorem \ref{thm:fullmain}, we will use the functions $r(t\pm\eps t,b)$ and $l(t\pm\eps t,b)$ for an arbitrary small $\eps>0$. This allow us to estimate $r(t\pm\eps t,b)$, where, for example, $b$ is given by \eqref{eq:examplesofS} with $\mu,\nu\in\R$ by $r(t\pm\tilde{\eps} t,b_0)$, where $b_0$ corresponds to $\mu=\nu=0$. Namely, we start with the following definition. 

\begin{definition}\label{def:logeqv}
Let $b_1,b_2:\R_+\to\R_+$ and, for some $\rho\geq0$, $b_i(x)>0$ for all $s\in[\rho,\infty)$, $i=1,2$. The functions $b_1$ and $b_2$ are said to be  
 {\em (asymptotically) log-equivalent}, if
 \begin{equation}\label{eq:log-equiv}
   \log b_1(x) \sim \log b_2(x), \quad x\to\infty.
 \end{equation}
\end{definition}

\begin{lemma}[\!\!\protect{\cite[Proposition 2.16]{FKT2016}}]\label{prop:etaepsforweaklyequiv}
Let $b_1,b_2:\R\to\R_+$ be two tail-decreasing functions which are log-equivalent, i.e. \eqref{eq:log-equiv} holds. Define
\begin{equation}\label{eq:relforeps}
\eta_\eps^\pm(t,b):=r(t\pm\eps t,b), \quad t>\tau.
\end{equation}
Then, for any $0<\eps_1<\eps<\eps_2<1$, there exists $\tau>0$, such that, for all $t\geq \tau$,
\begin{equation}
\eta_{\eps_2}^-(t,b_2) \leq \eta_{\eps}^-(t,b_1) \leq \eta_{\eps_1}^-(t,b_2) \leq \eta_{\eps_1}^+(t,b_2) \leq \eta_{\eps}^+(t,b_1) \leq \eta_{\eps_2}^+(t,b_2).\label{eq:etaineqeps}
\end{equation}
\end{lemma}
Clearly, replacing $b(x)$ on $b(-x)$ in \eqref{eq:relforeps}, one gets an analogue of \eqref{eq:etaineqeps} for $l(t\pm\eps t)$.

\begin{proposition}\label{eq:crucialforlinearshift}
        Let $b:\R\to\R_+$ be (right-side) long-tailed, tail-decreasing, and tail-convex.
        Then for any $0<\eps_1<\eps_2<1$ and $k>0$, there exists $\tau = \tau(k, \eps_1, \eps_2)> 0$, such that
        \begin{equation}\label{eq:front_lin_peturb}
                r(t-\eps_1 t,b) \geq r(t-\eps_2 t,b) + k t, \qquad t\geq \tau.
        \end{equation}
\end{proposition}
\begin{proof}
Since $b$ is decreasing and convex on $(\rho,\infty)$ for some $\rho>0$, it is well-known that the inverse function $b^{-1}$ is also convex on $(0,\alpha)$ for some $\alpha>0$. Since $t\mapsto e^{-\beta(1-\eps)t}$ is also a convex function, we conclude that, for each $\eps\in(0,1)$, the function $[\tau',\infty)\ni t\mapsto \eta(t):=r(t,b)$ is convex (for big enough $\tau'>0$). Prove that 
\[
f(t):=\eta\bigl((1-\eps_1)t\bigr)-\eta\bigl((1-\eps_2)t\bigr), \quad t\geq\tau'
\]
is a non-decreasing function. Indeed, since $\eta(\cdot)$ is convex, we have that the function $\dfrac{\eta(t)-\eta(s)}{t-s}$, $t,s\geq\tau'$, is non-decreasing in each of coordinates. Therefore, for each $t_1>t_2>\tau'$, we have $(1-\eps_1)t_1>(1-\eps_2)t_2$ and then
\begin{gather*}
\dfrac{\eta\bigl((1-\eps_2)t_1\bigr)-\eta\bigl((1-\eps_2)t_2\bigr)}{(1-\eps_2)(t_1-t_2)}\leq
\dfrac{\eta\bigl((1-\eps_1)t_1\bigr)-\eta\bigl((1-\eps_2)t_2\bigr)}{(1-\eps_1)t_1-(1-\eps_2)t_2}\\
\leq
\dfrac{\eta\bigl((1-\eps_1)t_1\bigr)-\eta\bigl((1-\eps_1)t_2\bigr)}{(1-\eps_1)(t_1-t_2)},
\end{gather*}
Multiplying this on $1-\eps_2\leq 1-\eps_1$, one gets
\[
\eta\bigl((1-\eps_2)t_1\bigr)-\eta\bigl((1-\eps_2)t_2\bigr)<\eta\bigl((1-\eps_1)t_1\bigr)-\eta\bigl((1-\eps_1)t_2\bigr),
\] 
that implies $f(t_1)>f(t_2)$. We set
\begin{equation}\label{eq:defnu}
\nu:=\inf_{t\geq \tau'}f(t)=f(\tau')>0.
\end{equation}

Since $b$ is long-tailed, one gets
        \begin{equation*}
                 \lim_{x\to\infty} \sup_{0\leq y\leq 1} \log\frac{b(x+y)}{b(x)} = \lim_{x\to\infty} \log \sup_{0\leq y \leq 1} \frac{b(x+y)}{b(x)} = 0.
        \end{equation*}
Therefore, for any $\delta>0$, there exists $x_0=x_0(\delta) \geq \rho$, such that
        \begin{equation}\label{eq:long_tailed_delta}
                \sup_{0\leq y\leq 1} (\log b(x) - \log b(x+y) ) \leq \delta, \qquad x\geq x_0.
        \end{equation}
        Let $\tau= \tau(\delta,\eps_1, \eps_2) \geq \tau'$ be such that $\eta\bigl((1-\eps_2)t\bigr) \geq x_0$, for all $t\geq \tau$. 
        For any fixed $t\geq \tau$, consider $N=N(t)$, such that
        \begin{equation*}
                \Delta := \frac{1}{N} f(t) \in \Bigl[\min\Bigl\{\nu, \frac{1}{2}\Bigr\}, 1\Bigr].
        \end{equation*}
        Then, by \eqref{eq:long_tailed_delta}, \eqref{eq:defnu}, for all $t\geq\tau$, one gets
        \begin{align*}
                &\quad (\eps_2-\eps_1)\beta t = \log b\bigl(\eta\bigl((1-\eps_1)t\bigr)\bigr) - \log b\bigl(\eta\bigl((1-\eps_2)t\bigr)\bigr) \\
                        &= \sum_{j=0}^{N-1} \Bigl(\log b\bigl(\eta\bigl((1-\eps_1)t\bigr)+j\Delta\bigr) - \log b\bigl(\eta\bigl((1-\eps_2)t\bigr)+\Delta+j\Delta\bigr) \Bigr) \\
                        &\leq \delta N \leq \frac{\delta}{\min\bigl\{\nu, \frac{1}{2}\bigr\}}f(t).
        \end{align*}
        Hence, for any $k> 0$, it is sufficient to choose $\delta \leq \frac{\beta\nu(\eps_2-\eps_1)}{k}$. The proof is fulfilled.
\end{proof}

Clearly, the corresponding analogue to \eqref{eq:front_lin_peturb} holds for $l(\cdot)$ as well.

\section{Convergence to positive stationary solution}\label{sec:convtotheta}

\begin{definition}\label{def:LSR+R}
\begin{enumerate}
        \item Let $\L(\R_+)$ denote the set of all right-side long-tailed, tail-decreasing and tail-continuous bounded functions $b:\R\to\R_+$. Let $\L(\R_-)$ be the set of all $b:\R\to\R_+$, such that $b(-x)$ belongs to $\L(\R_+)$. We set also
        \[
        \L(\R):=\L(\R_+)\cap\L(\R_-).
        \]

        \item Let $\P(\R_-)$ denote the set of all bounded functions $b:\R\to\R_+$ such $\inf\limits_{x\leq-\rho}b(x)>0$ for some $\rho>0$. Let $\P(\R_+)$ denote the set of all bounded functions $b:\R\to\R_+$ such $\inf\limits_{x\geq\rho}b(x)>0$ for some $\rho>0$. 
        We set
        \[
        \PL(\R):=\P(\R_-)\cap\L(\R_+).
        \]
\end{enumerate}
\end{definition}

Let \eqref{assum:kappa>m} hold. 
It will be convenient for us to extend Definition~\ref{def:ltrt} by setting
\begin{equation}\label{eq:ltrtinf}
 l(t,b):=\infty, \quad b\in\P(\R_-), \qquad 
r(t,b):=\infty, \quad b\in\P(\R_+),
 \end{equation}
for $t>\tau$ with a needed $\tau>0$.

\begin{proposition}\label{prop:subsoltolinear}
        Let \eqref{assum:kappa>m} hold. Let $b:\R\to\R_+$ be such that $b\in\L(\R_+)\cup \P(\R_+)$ and $b\in\L(\R_-)\cup \P(\R_-)$. Let $\eps\in(0,1)$ be fixed, and $\tau=\tau(\eps)>0$ be such that both
        \begin{equation}\label{eq:doprltrt}
         l_t:= l(t-\eps t,b), \qquad r_t:=r(t-\eps t,b)
         \end{equation}  
        are well-defined for $t>\tau$;  cf. also \eqref{eq:ltrtinf}. Let $\la>0$ be arbitrary. For $t>\tau$ and $x\in\R$, we define
\begin{equation}
                g(x,t) =g_\eps(x,t):=\la \1_{(-l_t,r_t)}(x)  +\la b(x)e^{\beta_\eps^- t}\1_{\R\setminus(-l_t,r_t)}(x)\in(0,\la].
                \label{eq:subsol_to_lin_eq_delta}
\end{equation}
Then, for each $\delta\in(0,\eps\beta)$, there exists $t_0=t_0(\eps,\delta)>\tau$, such that, for all $t\geq t_0$, the function $g$
        is a sub-solution to the equation
        \begin{equation*}
                \diff{v}{t}(x,t) = \ka (a*v)(x,t) - (m+\delta)v(x,t).
        \end{equation*}
        Namely, for all $t\geq t_0$ and $x\in\R$,
        \begin{equation}\label{eq:F_def}
                (\F[\delta] g)(x,t) := \diff{g}{t}(x,t) - \ka (a*g)(x,t) + (m+\delta)g(x,t) \leq 0.
        \end{equation}
\end{proposition}
\begin{proof}
        It is sufficient to prove \eqref{eq:F_def} for $x\geq 0$; indeed, then the result for $x<0$ may be obtained by replacing $b(x)$ on $b(-x)$. Since $b$ is long-tailed, \eqref{eq:uniformlongtailedR} yields that, for any $\delta_1\in\bigl(0, \frac{\beta\eps-\delta}{\ka}\bigr)$, there exists $x_0=x_0(\delta_1)$, such that 
        \begin{equation}\label{eq:b_over_b_geq_delta1}
                \sup_{|y|\leq h(x)} \frac{b(x+y)}{b(x)} \geq 1-\delta_1,\qquad x\geq x_0.
        \end{equation}
        In the sequel, to keep unified notations, we assume that both $h(r_t)$ and $r_t-h(r_t)$ are equal to $\infty$ when $r_t=\infty$, $t>\tau$, cf. \eqref{eq:ltrtinf} (remember that, by Lemma~\ref{le:basicforb}, $h(x)<\frac{x}{2}$).

        Note also, that, by the above,
                \begin{equation}\label{eq:rtltinf}
        \lim\limits_{t\to\infty} r_t = \lim\limits_{t\to\infty} l_t = \infty
        \end{equation}
        (it may be, see \eqref{eq:ltrtinf}, that either of, or both, $r_t$ and $l_t$ are equal to $\infty$ for all $t>\tau$).

        Prove, first, that there exists $t_0=t_0(\eps,\delta)>\tau$, such that
        \begin{equation}\label{eq:provethisnice}
        \frac{(a*g)(x,t)}{g(x,t)}\geq (1- \delta_1) \int_{-h(r_t)}^{l_t} a(y) dy
        \end{equation}
        for all $x\geq0$ and $t\geq t_0$. Note that, clearly,
        \begin{equation}
                (a*g)(x,t) \geq \int_{-h(r_t)}^{l_t} a(y) g(x-y,t)dy  \label{eq:aplus_conv_g}
        \end{equation}
        for $x\in\R$ and $t>\tau$.

1. Let $x\in[0,r_t-h(r_t))$, $t>\tau$. Then $-h(r_t)\leq y\leq l_t$ yields $-l_t\leq x-y<r_t$ and hence, by \eqref{eq:aplus_conv_g}, \eqref{eq:subsol_to_lin_eq_delta},
\begin{align*}
\frac{(a*g)(x,t)}{g(x,t)} \geq \frac{1}{\la }\la\int_{-h(r_t)}^{l_t} a(y) dy,
\end{align*}
that implies \eqref{eq:provethisnice}. If $b \in \L(\R_+)$, i.e. if $r_t<\infty$ for $t>\tau$, then we consider also two other possibilities.

2. Let $x\in[r_t-h(r_t),r_t)$, $t>\tau$. Then it is straightforward to get from \eqref{eq:aplus_conv_g} and \eqref{eq:subsol_to_lin_eq_delta}, that
\begin{equation}
\frac{(a*g)(x,t)}{g(x,t)} \geq  \int_{x-r_t}^{l_t} a(y) dy+ \int_{-h(r_t)}^{x-r_t} a(y) \frac{b(x-y)}{b(r_t)}dy,\label{eq:intermediate}
\end{equation}
where we used also that $b(r_t) = e^{-\beta_\eps^- t}$ for $t>\tau$. 
Next, for the considered $x$, $-h(r_t)\leq y\leq x-r_t$ yields $0\leq x-y-r_t<h(r_t)$, 
and hence, by \eqref{eq:b_over_b_geq_delta1}, there exists $t_1>\tau$ such that for all $t\geq t_1$ and $x\in[r_t-h(r_t),r_t]$
\[
\frac{b(x-y)}{b(r_t)}=\frac{b\bigl(r_t+(x-y-r_t)\bigr)}{b(r_t)}\geq 1- \delta_1,
\]
that, together  with \eqref{eq:intermediate}, implies \eqref{eq:provethisnice}.

3. Let $x\geq r_t$, $t>\tau$. Then, by \eqref{eq:aplus_conv_g} and \eqref{eq:subsol_to_lin_eq_delta},
\begin{equation}
\frac{(a*g)(x,t)}{g(x)} \geq \frac{\la e^{-\beta_\eps^- t}}{\la b(x)} \int_{x-r_t}^{l_t} a(y) dy
+\int_{-h(r_t)}^{x-r_t} a(y) \frac{b(x-y)}{b(x)} dy.\label{eq:intermediate2}
\end{equation}
Next, $e^{-\beta_\eps^- t}=b(r_t)\geq b(x)$ for $t>\tau$, since $b$ is decreasing on $[r_t,\infty)$. The latter also implies that $b(x-y)\geq b(x)$ if only $0\leq y\leq x-r_t$.
Finally,  by \eqref{eq:b_over_b_geq_delta1}, there exists $t_2>t_1$, such that
$b(x-y) \geq (1-\delta_1)b(x)$,
if only $-h(r_t)\leq y<0$, $x\geq r_t$, $t\geq t_2$. As a result, \eqref{eq:intermediate2} implies \eqref{eq:provethisnice}, which is proved hence for all $x\geq0$ and $t\geq t_2$.

Note that, by \eqref{eq:subsol_to_lin_eq_delta}, 
        \begin{equation}
        \diff{g}{t}(x,t) = \beta_\eps^- b(x) e^{\beta_\eps^- t} \1_{\R\backslash (-l_t,r_t)}(x)\leq \beta_\eps^- g(x,t). \label{eq:dg_dt}
        \end{equation}
Then, combining \eqref{eq:provethisnice} and \eqref{eq:dg_dt}, one gets
\begin{equation}
-\frac{ (\F[\delta]g)(x,t) }{ g(x,t) } \geq -\beta_\eps^- +\ka(1- \delta_1) \int_{-h(r_t)}^{l_t} a(y) dy -(m+\delta)\label{eq:temp}
\end{equation}
for all $x\geq0$ and $t\geq t_2$. By \eqref{eq:rtltinf}, we have that
\[
-\beta_\eps^-+\ka(1- \delta_1) \int_{-h(r_t)}^{l_t} a(y) dy -(m+\delta)\to \eps(\ka-m)-\delta_1\ka-\delta>0,
\]
as $t\to\infty$. Combining this with \eqref{eq:temp}, we conclude that there exists $t_0\geq t_2$, such that \eqref{eq:F_def} holds for $x\geq0$ and $t\geq t_0$.
        \end{proof}

\begin{remark}\label{rem:doesnotdependonlambda}
It is worth noting that, indeed, in the proof of Proposition~\ref{prop:subsoltolinear}, both $t_1$ and $t_2$ and hence $t_0$ do not depend on $\la$.
\end{remark}

\begin{proposition}\label{prop:subsoltononlinear}
Let \eqref{assum:kappa>m}--\eqref{assum:Glipschitz} hold and $b:\R\to\R_+$ be such that $b\in\L(\R_+)\cup \P(\R_+)$ and $b\in\L(\R_-)\cup \P(\R_-)$. Then, for each $\eps\in(0,1)$, there exist $\la_0=\la_0(\eps)\in(0,\theta)$ and $\tau_0=\tau_0(\eps)>0$, such that, for each $\la\in(0,\la_0)$, the function $g=g(x,t)$, given by \eqref{eq:subsol_to_lin_eq_delta}, is a sub-solution to \eqref{eq:basicequation}. Namely, for all $t\geq \tau_0$ and $x\in\R$,
        \begin{multline}\label{eq:F_defnl}
                (\F g)(x,t) := \diff{g}{t}(x,t) - \ka(a*g)(x,t) \\ + mg(x,t)+g(x,t)(Gg)(x,t) \leq 0.
        \end{multline}
\end{proposition}
\begin{proof}
Take any $\eps\in(0,1)$ and $\delta\in(0,\eps\beta)$. By \eqref{assum:kappa>m}--\eqref{assum:Glipschitz}, there exists $\la_0=\la_0(\delta)=\la_0(\eps)\in(0,\theta)$, such that $0\leq u\leq \la_0$ implies 
\begin{equation}\label{eq:smallG}
0\leq Gu\leq \delta.
\end{equation}
By~Proposition~\ref{prop:subsoltolinear} and Remark~\ref{rem:doesnotdependonlambda}, for each $\la\in(0,\la_0]$, the function $g=g(x,t)$, given by \eqref{eq:subsol_to_lin_eq_delta}, satisfies \eqref{eq:F_def} for all $x\in\R$ and $t>\tau_0$ for some $\tau_0>0$. Since \eqref{eq:subsol_to_lin_eq_delta} yields $g\leq\la_0$, then \eqref{eq:smallG} and \eqref{eq:F_def} imply
\begin{align*}
-\F g &=-\diff{g}{t}+\ka(a*g) - mg -g (Gg)\\ &\geq -\diff{g}{t}+\ka(a*g) - mg -\delta g =-\F[\delta] g\geq0,
\end{align*}
that yields \eqref{eq:F_defnl}.
\end{proof}

To proceed further we will need the following generalization of the comparison \eqref{eq:compofsol} for solutions to \eqref{eq:basicequation}.

\begin{lemma}[\!\!\!{\cite[Theorems 2.2]{FT2017a}}]\label{lem:comparison}
        Let \eqref{assum:kappa>m}--\eqref{assum:sufficient_for_comparison} hold.  Let $T>0$ be fixed. Suppose that  $u_1,u_2:[0,T]\to E$ are continuous mappings, continuously differentiable in $t\in(0,T]$, and such that, for $(x,t)\in\X \times(0,T]$,
    \begin{gather*}
                \frac{\partial u_1}{\partial t} - \ka a*u_1 +mu_1 +u_1Gu_1 \leq \frac{\partial u_2}{\partial t} - \ka a*u_2 +mu_2 +u_2Gu_2,\\
u_1(x,t)\geq0, \qquad u_2(x,t)\leq \theta,\\
      0\leq u_{1}(x,0)\leq u_{2}(x,0)\leq \theta.
    \end{gather*}
        Then $u_{1}(x,t)\leq u_{2}(x,t)$ for $(x,t)\in\X \times[0,T]$. 
\end{lemma}

\begin{theorem}\label{thm:bb:est_below}
Let either \eqref{assum:kappa>m}--\eqref{assum:improved_sufficient_for_comparison} hold or \eqref{assum:approx_of_basic} hold.  Let $0\leq u_0\leq\theta$, $u_0\not\equiv0$ (cf. Remark~\ref{rem:agreementnonzero}), and let $u=u(x,t)$ be the corresponding solution to~\eqref{eq:basicequation}. Suppose also that there exist $b:\R\to\R_+$ and $D,\rho>0$, such that 
\begin{enumerate}
        \item either \eqref{eq:intcase} holds, $b\in\L(\R)$, the inequality
\begin{equation}\label{eq:condonparambelow}
  (a*u_0)(x)\geq D b(x), 
\end{equation}
holds for all $|x|>\rho$, and $b$ is convex on $(-\infty,\rho)$ and on $(\rho,\infty)$.
\item or \eqref{eq:moncase} holds, $b\in\PL(\R)$, the inequality \eqref{eq:condonparambelow} holds for all $x>\rho$, and  $b$ is convex on $(\rho,\infty)$.
\end{enumerate}
Then, for each $\eps\in(0,1)$, 
  \begin{equation}\label{eq:insidethefront}
    \lim\limits_{t\to\infty}\essinf\limits_{x\in\La_\eps^-(t,b)} u(x,t)=\theta,
  \end{equation} 
  where
  \begin{equation}\label{eq:defLambda}
  \La_\eps^-(t,b):=\begin{cases}
[-l(t-\eps t,b), r(t-\eps t,b)], &  b\in\L(\R),\\[1mm]
(-\infty, r(t-\eps t,b)], &  b\in\PL(\R).
\end{cases}
  \end{equation}
\end{theorem}

\begin{proof}
First, we note that, by Proposition~\ref{thm:existandcompared},  $0\leq u_0\leq \theta$ implies $0\leq u(\cdot,t)\leq \theta$ for $t>0$.

 Let \eqref{eq:intcase} hold and $b\in\L(\R)$. Since $u_0\not\equiv 0$ in the sense of Remark~\ref{rem:agreementnonzero}, there exists a continuous function $\tilde{u}_0:\R\to\R_+$, such that $\tilde{u}_0(x)\leq u_0(x)$, $x\in\R$ and $\tilde{u}_0(x)\geq \delta$ for all $x\in B_\rho(x_0)$ with some $x_0\in\R$, $\delta,\rho>0$. Let $\tilde{u}(x,t)$ be the corresponding solution to \eqref{eq:basicequation}. Then by \cite[Theorem~2.1]{FT2017a}, $\tilde{u}(\cdot,t)$ is a continuous function for all $t>0$. We set also $I_\rho:=[-\rho,\rho]$.

 Let \eqref{eq:moncase} hold and $b\in\PL(\R)$. Then there exists a non-increasing continuous function $\tilde{u}_0:\R\to\R_+$ which is strictly decreasing on $(-\infty,-\rho)$ for some $\rho>0$, such that $\tilde{u}_0(x)\leq u_0(x)$, $x\in\R$ and
 $\tilde{u}_0(x)\geq \delta$, $x<-\rho$ for some $\delta>0$. Let $\tilde{u}(x,t)$ be the corresponding solution to \eqref{eq:basicequation}. Then by \cite[Theorem~2.1, Proposition~5.7]{FT2017a}, $\tilde{u}(\cdot,t)$ is a continuous and non-increasing function for all $t>0$. We set then $I_\rho:=(-\infty,\rho]$.

In both cases, by Proposition~\ref{thm:existandcompared}, 
\begin{equation}\label{eq:dop234432234}
\tilde{u}(x,t)\leq u(x,t), \quad x\in\X, \ t\geq0.
\end{equation}
 Moreover, by \cite[Proposition~5.3]{FT2017a}, 
\begin{equation}\label{eq:wwqeerqw324342}
\tilde{u}(x,t)>\inf_{\substack{y\in\R\\s>0}}\tilde{u}(y,s)\geq0, \quad x\in\R,\ t>0.
\end{equation}

Fix an arbitrary $\eps\in(0,1)$ and take any $\delta\in(0,\eps)$. Consider
 $\la_0=\la_0(\delta)>0$ and $\tau_0=\tau_0(\delta)>0$, both given by~Proposition~\ref{prop:subsoltononlinear}. Then, by \eqref{eq:wwqeerqw324342},
\[
\ga:=\min_{x\in I_\rho}\tilde{u}(x,\tau_0)>0,
\]
and, by \eqref{eq:dop234432234}, 
\begin{equation}\label{eq:sawwqewerqw}
u(x,\tau_0)\geq \ga, \quad x\in I_\rho.
\end{equation}

By \eqref{assum:Gpositive}, $0\leq u\leq \theta$ implies $Gu\leq\beta$. Rewrite \eqref{eq:basicequation} in the form \eqref{eq:RDequation} with $F$ given by \eqref{eq:FthroughG}, then, by \eqref{eq:propertiesofF}, $Fu\geq0$. Then, it is straightforward to show by  Duhamel's principle (see \cite[formula (4.16)]{FKT2016}), that, for all $t>0$, 
\begin{equation}\label{eq:sawwqewerqw2}
u(x,t)\geq \ka t e^{-\ka t}(a*u_0)(x), \quad x\in\X.
\end{equation}

Let us re-define the given function $b$ by setting $b(x):=\frac{\ga}{D}$ for $x\in I_\rho$. Note that, by Definition~\ref{def:LSR+R}, the re-defined function will still belong to either $\L(\R)$ or $\PL(\R)$, and, by Definition~\ref{def:ltrt} and \eqref{eq:defLambda}, for big enough $t$, the set $\La_\eps^-(t,b)$ will remain the same for the new $b$. For the new $b$, \eqref{eq:condonparambelow}, \eqref{eq:sawwqewerqw}, \eqref{eq:sawwqewerqw2} yield
\begin{equation}
u(x,\tau_0)\geq D \ka t e^{-\ka t}b(x), \quad x\in\R.\label{eq:coolest}
\end{equation}

Next, combining again \eqref{eq:dop234432234} and \eqref{eq:wwqeerqw324342}, we will get, cf. \eqref{eq:explicit},  \eqref{eq:explicit2}, \eqref{eq:doprltrt}, \eqref{eq:defLambda},
\begin{equation}\label{eq:deflambda1}
\la_1:=\essinf_{x\in  \La_\delta^-(\tau_0,b)} u(x,\tau_0)>0.
\end{equation}
Set now
\[
\la:=\min\bigl\{\la_0, \la_1, D\ka \tau_0 e^{-(\ka+\beta_{\delta}^-) \tau_0}\bigr\}.
\]
Then, by \eqref{eq:coolest}, \eqref{eq:deflambda1}, and \eqref{eq:subsol_to_lin_eq_delta}, we have, for a.a.~$x\in\X$,
\begin{equation}\label{eq:saqrwrwq}
u(x,\tau_0)\geq g_\delta(x,\tau_0).
\end{equation}
Since, by Proposition~\ref{prop:subsoltononlinear}, $g_\delta $ is a sub-solution to \eqref{eq:basicequation}, we immediately conclude from Lemma~\ref{lem:comparison} and \eqref{eq:saqrwrwq}, that, for each $\tau\geq0$,
\begin{equation*}
  u(x,\tau_0+\tau)\geq g_{\delta}(x,\tau_0+\tau), \quad \text{for a.a. } x\in\X.
\end{equation*}
In particular, cf. \eqref{eq:subsol_to_lin_eq_delta}, \eqref{eq:defLambda},
\begin{equation}\label{eq:ineqdopevident}
  u(x,\tau_0+\tau)\geq \la \quad  \text{for a.a.} \ x\in\La_\delta^-(\tau_0+\tau,b), \ \tau\geq0.
\end{equation}

Fix an arbitrary $\tau\geq0$, such that 
\[
r((\tau_0+\tau)(1- \delta),b)>2, \qquad l((\tau_0+\tau)(1- \delta),b)>2;
\] 
for the latter, see also \eqref{eq:ltrtinf} in the case $b\in\PL(\R)$. Set  
\[
\widetilde{\Lambda}:=\begin{cases}
\bigl[-l((\tau_0+\tau)(1- \delta),b)+1, r((\tau_0+\tau)(1- \delta),b)-1\bigr], &b\in\L(\R),\\[2mm]
\bigl(-\infty, r((\tau_0+\tau)(1- \delta),b)-1\bigr], &b\in\PL(\R).
\end{cases}
\]
Clearly,
\begin{equation}\label{eq:evidentunion}
  \La_\delta^-(\tau_0+\tau,b)=\bigcup_{y\in \widetilde{\Lambda}} B_1(y).
\end{equation}
Take and fix now an arbitrary $y\in\widetilde{\La}$.
Then, by~\eqref{eq:ineqdopevident},
\begin{equation}\label{eq:aswrqqrw3qr3w}
u(x,\tau_0+\tau)\geq \la \1_{B_1(y)}(x), \quad x\in\X.
\end{equation}
Consider the equation \eqref{eq:basicequation} with the initial condition $v_0(x)=\la \1_{B_1(y)}(x)$, $x\in\X$; let $v(x,t)$ be the corresponding solution to \eqref{eq:basicequation}.
By the uniqueness and comparison \eqref{eq:compofsol} in Proposition~\ref{thm:existandcompared}, \eqref{eq:aswrqqrw3qr3w} yields
\begin{equation}\label{eq:qrwerq2}
u(x,\tau_0+\tau+t)\geq v(x,t), \qquad x\in\R,\ t\in\R_+.
\end{equation}

Let, first, \eqref{assum:kappa>m}--\eqref{assum:improved_sufficient_for_comparison} hold. Take an arbitrary $\mu\in(0,\theta)$. Apply Proposition~\ref{thm:hairtrigger} to the solution $v$ and $K=B_1(y)$; then there exists $t_\mu\geq 1$, such that $v(x+t_\mu\ma,t_\mu)\geq \mu$ for a.a. $x\in B_1(y)$.
As a result, by \eqref{eq:qrwerq2},
\begin{equation}\label{eq:ineqonemore}
u(x+t_\mu\ma,\tau_0+\tau+t_\mu)\geq \mu,
\end{equation}
for each $\tau\geq0$ and a.a.~$x\in B_1(y)$. 
Stress that, by \eqref{assum:G_commute_T}, $t_\mu$ does not depend on a $y\in\R$; therefore, beside $y\in\widetilde{\Lambda}=\widetilde{\Lambda}(\tau)$, $t_\mu$ does not depend on $\tau$. As a result, by \eqref{eq:evidentunion} for any $\delta\in(0,1)$ and $\mu\in (0,\theta)$, there exist $\la_0=\la_0(\delta)>0$, $\tau_0=\tau_0(\delta)>0$, and $t_\mu\geq 1$, such that, for all $\tau\geq 0$ and for a.a.~$x\in\La_\delta^-(\tau_0+\tau,b)$, the inequality \eqref{eq:ineqonemore} holds. 

Take any $\tilde{\eps}\in(\delta,\eps)$. Apply now \cite[Lemma 3.1]{FKT2016} for $\eps_2:=\tilde{\eps}>\delta=:\eps_1$, $t_1=\tau_0$, $t_2=\tau_0+t$; cf. also \eqref{eq:relforeps}. One gets that there exists $\tau_1\geq0$, such that, for all $\tau\geq \tau_1$,
\[
r((\tau+\tau_0+t_\mu)(1-\tilde{\eps}),b)\leq r((\tau+\tau_0)(1- \delta),b),
\]
and the same inequality holds for $l(\cdot,b)$. 
As a result, \eqref{eq:ineqonemore} holds for all $\tau\geq \tau_1$ and a.a.~$x\in\La_{\tilde{\eps}}^-(\tau_0+\tau+t_\mu,b)\subset\La_\delta^-(\tau_0+\tau,b)$.

In particular, for all $\tau>0$,
\begin{equation}\label{eq:ineqonemore22}
u(x,\tau_0+\tau+t_\mu)\geq \mu,
\end{equation}
provided that
\[
-l((\tau_0+\tau+t_\mu)(1- \tilde{\eps}),b)-t_\mu\ma <x<r((\tau_0+\tau+t_\mu)(1- \tilde{\eps}),b)-t_\mu\ma,
\]
cf. also \eqref{eq:ltrtinf} for the case $b\in\PL(\R)$. Denote $T:=\tau_0+\tau+t_\mu$. Let $\ma\geq 0$ (the opposite case may be considered analogously). Then, in particular, \eqref{eq:ineqonemore22} holds for all
\begin{equation*}
-l(T(1- \eps),b) <x<r(T(1- \tilde{\eps}),b)-T\ma,
\end{equation*}
as $l(\cdot,b)$ is increasing. By Proposition~\ref{eq:crucialforlinearshift},
\[
r(T(1- \eps),b)\leq r(T(1- \tilde{\eps}),b)-T\ma
\]
for $T$ big enough. As a result,  \eqref{eq:ineqonemore22} holds for all
\[
-l(T(1- \eps),b) <x<r(T(1-\eps),b),
\]
and big enough $T$. 
In other words, we have then that \eqref{eq:ineqonemore22} holds for all $x\in\La_\eps^-(\tau_0+\tau+t_\mu,b)$ and $\tau>\tau_2$ for some $\tau_2>\tau_1$. Since $\mu\in(0,\theta)$ was arbitrary, the latter fact yields \eqref{eq:insidethefront}. 

Let now \eqref{assum:approx_of_basic} hold. Then, for a sufficiently large $n\in\N$, we will take an arbitrary $\mu\in\bigl(0,\theta- \frac{1}{n}\bigr)$, and, using \eqref{eq:hair-trig-mod} and the same arguments as the above, we will show that \eqref{eq:ineqonemore22} holds for all $x\in\La_\eps^-(\tau_0+\tau+t_\mu,b)$ and big enough $\tau$. Then, the arbitrariness of $n$ and $\mu$ yields \eqref{eq:insidethefront} as well.
\end{proof}

The following result is a simple modification of \cite[Proposition 3.17]{FKT2016}.
\begin{proposition}\label{prop:liminfbelow}
  Let $f\in L^1(\X,\R_+)$ and $g\in\L(\R)\cup\PL(\R)$. Then there exists $D>0$ such that
  \begin{equation}\label{eq:3223352csaerw}
  (g*f)(x)\geq D g(x)
  \end{equation}
  for all $|x|\geq\rho$ if $g\in\L(\R)$ and for all $x>\rho$ if $g\in\PL(\R)$.
\end{proposition}
\begin{proof}
 For any $r>0$, we have that
  \begin{align*}
    \frac{(g*f)(x)}{g(x)}&\geq \int_{|y|\leq r}\frac{g(x-y)}{g(x)}f(y)\,dy\\
    &\geq \biggl(1-\sup_{|y|\leq r}\Bigl\lvert
    \frac{g(x-y)}{g(x)}-1\Bigr\rvert \biggr)\int_{|y|\leq r}f(y)\,dy.
  \end{align*}
  By Lemma~\ref{le:basicforb}, item 2, and Definition~\ref{def:LSR+R}, one  gets that the latter expression in brackets converges to $0$ as $x\to\pm\infty$ for the case $g\in\L(\R)$ or $x\to\infty$ for the case $g\in \PL(\R)$. Therefore, there exists $D>0$, such that $\frac{(g*f)(x)}{g(x)}\geq D$.
\end{proof}

\begin{corollary}\label{cor:below}
Let either \eqref{assum:kappa>m}--\eqref{assum:improved_sufficient_for_comparison} hold or \eqref{assum:approx_of_basic} hold.  Let $u_0\in E_\theta^+$, $u_0\not\equiv0$, cf.~Remark~\ref{rem:agreementnonzero}; and let $u$ be the corresponding solution to \eqref{eq:basicequation}. 
\begin{enumerate}
        \item Let $u_0\in L^1(\R)$. Suppose that, for some $b\in\L(\R)$, $\rho>0$,
                \begin{equation*}
                        \text{either} \quad u_0(x)\geq b(x) \quad \text{or} \quad  a(x)\geq b(x), \qquad |x|>\rho.
                \end{equation*}
        Then \eqref{eq:insidethefront} holds.
        \item Let $u_0$ is non-increasing on $\R$, and $\lim_{x\to\infty}u_0(x)=0$. 
Suppose that, for some $b\in\PL(\R)$, $\rho>0$, either
        \begin{equation*}
        u_0(x)\geq b(x), \quad x>\rho,
        \end{equation*}
        or, for some $\delta>0$, $x_0\in\R$, $u_0(x)\geq \delta\1_{(-\infty,x_0)}(x)$, $x\in\X$, and   
                \begin{equation}\label{eq:qweqwqw22231er}
                        \int_x^\infty a(y)\,dy\geq b(x), \quad x>\rho. 
                \end{equation}
        Then \eqref{eq:insidethefront} holds.    
\end{enumerate}
\end{corollary}
\begin{proof}
\begin{enumerate}
        \item The statement is a straightforward application of Theorem~\ref{thm:bb:est_below} and inequality \eqref{eq:3223352csaerw}, applied for either 
        \[
        f=a, \quad g=u_0\1_{(-\rho,\rho)}+b\1_{\R\setminus(-\rho,\rho)}\in\L(\R)
        \] 
        or
        \[
        f=u_0, \quad g=a\1_{(-\rho,\rho)}+b\1_{\R\setminus(-\rho,\rho)}\in\L(\R)
        \]
        \item The first case is also followed from  Theorem~\ref{thm:bb:est_below} and inequality \eqref{eq:3223352csaerw} with 
        \[
        f=a, \quad g=u_0\1_{(-\infty,\rho)}+b\1_{[\rho,\infty)}\in\PL(\R)
        \] 
        $f=a$ and $g=u_0\in \PL(\R)$. The second case follows from the following chain of inequalities: first, because of \eqref{eq:qweqwqw22231er} and \eqref{eq:longtaileddef},
        \[
        (a*u_0)(x)\geq \delta \int_{x-x_0}^\infty a(y)\,dy\geq \delta b(x-x_0), \quad x>\rho+x_0,
        \]
        (assuming, without loss of generality that $\rho+x_0>0$);
        and, because of \eqref{eq:longtaileddef}, for a small $\delta'>0$,
        \[
                b(x-x_0)\geq (1- \delta')b(x), \quad x>\rho_{\delta'}
        \]
        for some $\rho_{\delta'}>\rho+x_0$.\qedhere
\end{enumerate}
\end{proof}

\section{Convergence to zero}

\begin{proposition}\label{prop:generalised_kestens}
                The following statements hold.
        \begin{enumerate}
                \item Let $u_0$ satisfy \eqref{eq:intcase}. Suppose that there exists $p\in L^1(\R)$, such that both $p(s)$, $p(-s)$ are sub-exponential on $\R_+$ in the sense of densities, and there exist $\rho, K>0$, such that
                        \begin{align}
                                &p(x+\tau)\leq Kp(x), \quad p(-x-\tau)\leq Kp(-x), &\quad &x\geq\rho,\ \tau\geq0, \label{eq:p_weakly_monotone}\\
                                &\max\bigl\{ a(x), \ u_0(x)\bigr\} \leq p(x), &\quad &|x|\geq\rho.
                        \end{align}
                        Then, for any $\eps\in(0,1)$, there exist $C_\eps$, $x_\eps>0$, such that
        \begin{equation}\label{eq:generalised_kestens_densities}
                (a^{*n}*u_0)(x) \leq C_\eps(1+\eps)^{n} p(x),\qquad |x|>x_\eps,\ n\in\N.
        \end{equation}
                \item Let $u_0$ satisfy \eqref{eq:moncase}  and be of a bounded variation on $\R$. Suppose that there exists $q\in L^\infty(\R)$ which is decreasing to $0$ on $\R$ and is sub-exponential on $\R_+$ in the sense of distributions, and there exists $\rho>0$, such that $q$ is continuous on $[\rho,\infty)$ and
                        \begin{align}
                                &\max\biggl\{ \int_x^\infty a(y)\,dy, \ u_0(x)\biggr\} \leq q(x), &\qquad &x\geq\rho.
                        \end{align}
                        Then, for any $\eps\in(0,1)$, there exist $C_\eps$, $x_\eps>0$, such that
        \begin{equation}\label{eq:generalised_kestens_distributions}
                (a^{*n}*u_0)(x) \leq C_\eps(1+\eps)^{n} q(x),\qquad x>x_\eps,\ n\in\N.
        \end{equation}
        \end{enumerate}
\end{proposition}

\begin{proof}
1) Define, for $x\in\R$,
        \begin{align*}
                \tilde{a}(x) &:= \1_{(-\infty,R)}(x) a(x) + \1_{[R,\infty)}(x) p(x), \\ 
                \tilde{u}_0(x) &:= \1_{(-\infty,R)}(x) u_0(x) + \1_{[R,\infty)}(x) p(x), 
        \end{align*}
        where $R=R(\eps)>$ is chosen such that $\max\bigl\{ \|\tilde{a}\|_1, \frac{\|\tilde{u}_0\|_1}{\|u_0\|_1}\bigr\} \leq \sqrt{1+\eps}$. (Here the sub-index $1$ denotes the norm in $L^1(\R)$.)   
        Then $a\leq \tilde{a}$, $u_0\leq \tilde{u}_0$, and $\tilde{a},\ \tilde{u}_0$ are sub-exponential on $\R_+$ in the sense of densities (cf. \cite[Corollary 2.18]{FT2017b}). 
        By~\eqref{eq:Kb1},  which we apply for $\delta=\sqrt{1+\eps}-1$, $C_\eps=\frac{\tilde{c}_\eps}{\sqrt{1+\eps}}$, $b(x) =\frac{1}{\|\tilde{a}\|_1}\tilde{a}(x)$, $x\in\R$, there exists $\tilde{x}_0 = \tilde{x}_0(\eps)>0$, such that
        \begin{equation}\label{eq:kesten_density}
                \tilde{a}^{*n}(x) \leq \tilde{c}_\eps (1+\eps)^n \tilde{a}(x),\qquad x\geq \tilde{x}_0.
        \end{equation}
        Let us estimate
        \begin{equation}\label{eq:wqwqwqewwqeweq}
        \begin{aligned}
                (a^{*n}*u_0)(x) &=  \int_{\tilde{x}_0}^{\infty} a^{*n}(y) u_0(x-y)dy + \int_{-\infty}^{\tilde{x}_0} a^{*n}(y)u_0(x-y)dy \\&=: I_1(x) + I_2(x).
        \end{aligned}
        \end{equation}
        By \eqref{eq:p_weakly_monotone}, \eqref{eq:kesten_density}, the following estimate holds for $x\geq \tilde{x}_0 + \max\{\rho, R\} =:\tilde{x}_1$,
        \begin{align*}
                I_1(x) &\leq \int_{\tilde{x}_0}^{\infty} \tilde{a}^{*n}(y) \tilde{u}_0(x-y)dy \leq  \int_{\tilde{x}_0}^{\infty} \tilde{c}_\eps (1+\eps)^n \tilde{a}(y) \tilde{u}_0(x-y) dy\\ 
                        &\leq  \tilde{c}_\eps (1+\eps)^n (\tilde{a}*\tilde{u}_0)(x);\\[2mm]
                I_2(x) &\leq \int_{-\infty}^{\tilde{x}_0} \tilde{a}^{*n}(y)\tilde{u}_0(x-y)dy = \int_{-\infty}^{\tilde{x}_0} \tilde{a}^{*n}(y)p(x-y)dy\\ 
                &\leq  Kp(x-\tilde{x}_0)\int_{-\infty}^{\tilde{x}_0} \tilde{a}^{*n}(y) dy \leq K (1+\eps)^{\frac{n}{2}} p(x-\tilde{x}_0).
        \end{align*}
        By \cite[Proposition 2.17]{FT2017b} and since $p$ is long-tailed, there exists $x_\eps\geq \tilde{x}_1$, such that for all $x\geq x_\eps$,
        \begin{align*}
                (\tilde{a}*\tilde{u}_0)(x) &\leq (1+\eps)2 \tilde{a}(x) =(1+\eps)2 p(x), \\
                        p(x-\tilde{x}_0) &\leq (1+\eps)p(x). 
        \end{align*}
        Hence, \eqref{eq:generalised_kestens_densities} holds for all $x>x_\eps$ and  $C_\eps := 2(1+\eps)\max\{\tilde{c}_\eps,K\}$. Redefining $x_\eps$ and $C_\eps$ if needed, we prove similarly \eqref{eq:generalised_kestens_densities} for all $x<-x_\eps$.

        2) Define, for $x\in\R$,
        \begin{align*}
                A(x)& := \int_{x}^{\infty} a(y)dy, \\ 
                \tilde{A}(x) &:= \1_{(-\infty,R)}(x) \biggl(\int_{x}^{R}a(y)dy + q(R)\biggr)+ \1_{[R,\infty)}(x) q(x),\\ 
                \tilde{u}_0(x) &:= \1_{(-\infty,R)}(x) \bigl(u_0(x)-u_0(R) + q(R)\bigr)+ \1_{[R,\infty)}(x) q(x),
        \end{align*}
        where $R=R(\eps)>0$ is such that $\max\bigl\{\tilde{u}_0(-\infty), \frac{\tilde{A}(-\infty)}{A(-\infty)}\bigr\} \leq \sqrt{1+\eps}$.
        Then $A\leq \tilde{A},\ u_0\leq\tilde{u}_0$, and $\tilde{A},\ \tilde{u}_0$ are sub-exponential on $\R_+$ in a sense of distributions.
        By \eqref{eq:Kb1distr}, which we apply for $\delta=\sqrt{1+\eps}-1$, $C_\delta =\frac{\tilde{c}_\eps}{\sqrt{1+\eps}} $, $b(x)=\frac{1}{\tilde{A}(-\infty)}\tilde{A}(x)$, $x\in\R$, there exists $\tilde{x}_0 = \tilde{x}_0(\eps)\geq 0$ (in fact, one can put $\tilde{x}_0=0$), such that
        \begin{equation}\label{eq:kesten_distribution}
                \int_{x}^{\infty} a^{*n}(y) dy = A^{\star n}(x) \leq \tilde{A}^{\star n}(x) \leq \tilde{c}_\eps (1+\eps)^n \tilde{A}(x),\qquad x\geq \tilde{x}_0.
        \end{equation}
        Let us estimate \eqref{eq:wqwqwqewwqeweq} for chosen $\tilde{x}_0$.
        By \eqref{eq:kesten_distribution} and \cite[Ch.I Theorem 4b]{Wid1941}, we have for $x\geq \tilde{x}_0 + \max\{\rho, R\} =:\tilde{x}_1$,
        \begin{align*}
                I_1(x) &= - \int_{\tilde{x}_0}^{\infty} u_0(x-y) d A^{\star n}(y) =  u_0(x-\tilde{x}_0) A^{\star n}(\tilde{x}_0) +\int_{\tilde{x}_0}^{\infty} A^{\star n} (y)d u_0(x-y) \\
                &\leq u_0(x-\tilde{x}_0) + \tilde{c}_\eps (1+\eps)^n \int_{\tilde{x}_0}^{\infty} \tilde{A}(y) d u_0(x-y)\\
                &\leq u_0(x-\tilde{x}_0) + \tilde{c}_\eps (1+\eps)^n \tilde{A}(\tilde{x}_0) u_0(x-\tilde{x}_0) - \tilde{c}_\eps (1+\eps)^n \int_{\tilde{x}_0}^{\infty} u_0(x-y) d \tilde{A}(y)\\
                &\leq q(x-\tilde{x}_0) + \tilde{c}_\eps(1+\eps)^n (q(x-\tilde{x}_0) + \tilde{A}^{\star 2}(x));\\
                I_2(x) &\leq q(x-\tilde{x}_0)\int_{-\infty}^{\tilde{x}_0} a^{*n}(y) dy \leq q(x-\tilde{x}_0).
        \end{align*}
                Since $\tilde{A}$ is sub-exponential and long-tailed, there exists $x_\eps\geq \tilde{x}_1$, such that for all $x\geq x_\eps$,
        \begin{align*}
                &\tilde{A}^{\star 2}(x) \leq 2 (1+\eps) \tilde{A}(x) = 2(1+\eps)q(x),\\
                &\tilde{A}(x-\tilde{x}_0) = q(x-\tilde{x}_0) \leq (1+\eps)q(x).
        \end{align*}
        Hence, \eqref{eq:generalised_kestens_distributions} holds for $C_\eps = (4+3\eps)\tilde{c}_\eps$, and the proof is fulfilled.
\end{proof}

\begin{definition}\label{def:SLSR+R}
\begin{enumerate}
        \item Let $\S(\R_+)\subset\L(\R_+)$ denote the set of all bounded functions $b:\R\to\R_+$ which are sub-exponential on $\R_+$ in the sense of densities, tail-decreasing and tail-continuous. Let $\S(\R_-)\subset\L(\R_-)$ be the set of all bounded $b:\R\to\R_+$, such that $b(-x)$ belongs to $\S(\R_+)$. We set also
        \[
        \S(\R):=\S(\R_+)\cap\S(\R_-).
        \]

        \item Let $\M(\R)$ denote the set of all bounded monotone functions $b:\R\to\R_+$ such that $\lim\limits_{x\to\infty}b(x)=0$. 
        We set also
        \[
        \MS(\R):=\M(\R)\cap\S(\R_+).
        \]
\end{enumerate}
\end{definition}

\begin{proposition}\label{prop:lin:bb:est_above}
Let \eqref{assum:kappa>m} hold and $0\leq u_0\in E$. Suppose that $u_0\not\equiv0$, cf. Remark~\ref{rem:agreementnonzero}. Let $w=w(x,t)$ be the corresponding solution to~\eqref{eq:linearequation}. Let either \eqref{eq:intcase} hold
and $b\in\S(\R)$ be such that
\begin{equation}\label{eq:dsadasdsad}
\max\bigl\{ a(x), \ u_0(x)\bigr\}\leq b(x), \quad |x|>\rho,
\end{equation}
or \eqref{eq:moncase} hold and $b\in\MS(\R)$ be such that
\begin{equation}\label{eq:dsadasdsad1}
\max\biggl\{ \int_x^\infty a(y)\,dy, \ u_0(x)\biggr\}\leq b(x), \quad x>\rho,
\end{equation}
for some $\rho>0$. 
Then, for each $\eps\in(0,1)$, 
  \begin{equation}\label{eq:outsidethefront-lin}
    \lim\limits_{t\to\infty}\essinf\limits_{x\in\R\setminus\La_\eps^+(t,b)} w(x,t)=0,
  \end{equation} 
  where
  \begin{equation}\label{eq:defLambdaplus}
  \La_\eps^+(t,b):=\begin{cases}
[-l(t+\eps t,b), r(t+\eps t,b)], &  b\in\S(\R),\\[1mm]
(-\infty, r(t+\eps t,b)], &  b\in\MS(\R).
\end{cases}
  \end{equation}
\end{proposition}

\begin{proof}
The solution to \eqref{eq:linearequation} is given by
\[
w(x,t)=e^{-mt}u_0(x)+e^{-mt}\sum_{n=1}^\infty \frac{(\ka t)^n}{n!}(a^{*n}*u_0)(x), \quad x\in\R.
\]
By Proposition~\ref{prop:generalised_kestens} with $p=b$ for $b\in\S(\R)$ or $q=b$ for $b\in\MS(\R)$, one gets that, for any $\delta\in(0,1)$, there exist $C_\delta$, $x_\delta>0$, such that, for all $|x|>x_\delta>\rho$, in the case $b\in\S(\R)$, and for all $x>x_\delta>\rho$, in the case $b\in\MS(\R)$,
\begin{align*}
 w(x,t)&\leq e^{-mt}u_0(x)+e^{-mt}\sum_{n=1}^\infty \frac{(\ka t)^n}{n!}C_\delta(1+\delta)^{n} b(x)\\
\intertext{and since, in both cases \eqref{eq:dsadasdsad} and \eqref{eq:dsadasdsad1}, $u_0(x)\leq b(x)$ for the considered values of $x$, one can continue}
&\leq \max\{C_\delta,1\} e^{\ka (1+\delta)t-mt}b(x).
\end{align*}
Take any $\eps\in(0,1)$ and $\delta'\in(0,\eps(\ka-m))$, and set $\delta=\frac{(\ka-m)\eps- \delta'}{\ka}\in(0,1)$, that ensures $\ka (1+\delta)-m=(\ka-m)(1+\eps)-\delta'$. Therefore,
\[
w(x,t)\leq \max\{C_\delta,1\} e^{-\delta't} e^{(\ka-m)(1+\eps)t}b(x), 
\]
again, for either all  $|x|>x_\delta$ or for all $x>x_\delta$, depending on the class to that $b$ belongs.
By \eqref{eq:defLambda} and Definition~\ref{def:ltrt}, we conclude then that there exists $\tau>0$, such that, for all $t>\tau$, 
\[
w(x,t)\leq \max\{C_\delta,1\} e^{-\delta't} , \quad x\in\R\setminus\La_\eps^+(t,b),
\]
that implies the statement.
\end{proof}

Now we can easily get the corresponding result for the solution to \eqref{eq:basicequation}.

\begin{theorem}\label{thm:bb:est_above}
Let \eqref{assum:kappa>m}--\eqref{assum:sufficient_for_comparison} hold.  Let $0\leq u_0\leq\theta$, $u_0\not\equiv0$ (cf. Remark~\ref{rem:agreementnonzero}), and let $u=u(x,t)$ be the corresponding solution to~\eqref{eq:basicequation}. Let $a,u_0$ and $b:\R\to\R_+$ satisfy the assumptions of Proposition~\ref{prop:lin:bb:est_above}. Then, for each $\eps\in(0,1)$, 
  \begin{equation}\label{eq:outsidethefront}
    \lim\limits_{t\to\infty}\essinf\limits_{x\in\R\setminus\La_\eps^+(t,b)} u(x,t)=0,
  \end{equation} 
  where $\La_\eps^+$ is given by \eqref{eq:defLambdaplus}.
\end{theorem}
\begin{proof}
By Proposition~\ref{thm:existandcompared},  $0\leq u_0\leq \theta$ implies $0\leq u(\cdot,t)\leq \theta$ for $t>0$; and then, by \eqref{assum:Gpositive}, $Gu\geq0$. Then, it is straightforward to show by Duhamel's principle, that $u(\cdot,t)\leq w(\cdot,t)$, $t>0$, where $w$ solves \eqref{eq:linearequation}. Hence the proof follows from Proposition~\ref{prop:lin:bb:est_above}.
\end{proof}

Finally, one can prove the main Theorem~\ref{thm:fullmain}.
\begin{proof}[Proof of \protect Theorem~\ref{thm:fullmain}]
Let $\eps\in(0,1)$, $b_1$ and $b_2$ be fixed. Take any $\eps_1\in(0,\eps)$.
Apply Corollary~\ref{cor:below} for the function $b_1$  and $\eps_1$, and apply Theorem~\ref{thm:bb:est_above} for the function $b_2$ and $\eps_1$. By \eqref{eq:eqandbdd3}, one can apply Lemma~\ref{prop:etaepsforweaklyequiv}, then, for big $t$,
\[
r(t-t\eps,b) \leq r(t-t\eps_1,b_1), \quad  r(t+t\eps_1,b_2) \leq r(t+t\eps,b),
\]
and the same holds for $l(\cdot)$. Therefore, for big $t$,
\[
\La_\eps^-(t,b)\subset \La_{\eps_1}^-(t,b_1),  \qquad \R\setminus\La_\eps^+(t,b)\subset\R\setminus\La_{\eps_1}^+(t,b_1),
\]
and hence
\[
\theta\geq\essinf_{x\in \La_\eps^-(t,b)}u(x,t)\geq \essinf_{x\in\La_{\eps_1}^-(t,b_1)}u(x,t)\to\theta, \quad t\to\infty
\]
and
\[
0\leq \esssup_{x\in\R\setminus\La_\eps^+(t,b)}u(x,t)\leq\esssup_{x\in\R\setminus\La_{\eps_1}^+(t,b_1)}u(x,t)\to0, \quad t\to\infty,
\]
that completes the proof.
\end{proof}

\section*{Acknowledgments}
Authors gratefully acknowledge the financial support by the DFG through CRC 701 ``Stochastic
Dynamics: Mathematical Theory and Applications'' (DF~and~PT), the European
Commission under the project STREVCOMS PIRSES-2013-612669 (DF), and the ``Bielefeld Young Researchers'' Fund through the Funding Line Postdocs:
``Career Bridge Doctorate\,--\,Postdoc'' (PT).

\let\OLDthebibliography\thebibliography
\renewcommand\thebibliography[1]{
  \OLDthebibliography{#1}
  \setlength{\parskip}{2pt}
  \setlength{\itemsep}{0pt plus 0.3ex}
}

\end{document}